\documentclass[10pt]{amsart}
\usepackage{amssymb}
\usepackage[english]{babel}
\usepackage{fancyhdr}
\usepackage{colortbl}
\usepackage{enumerate}
\usepackage{graphicx} 
\usepackage{float}
\usepackage{mathrsfs}
\usepackage{cite}
\usepackage{kpfonts}
\usepackage{tcolorbox} 
\usepackage{enumitem}
\usepackage{bbm}
\usepackage{amsmath}
\usepackage{amsfonts}
\usepackage{amsthm} 
\usepackage{latexsym}
\usepackage{mathtools}           
\usepackage[linesnumbered, ruled, vlined]{algorithm2e}
\usepackage[left=2.3cm,right=2.3cm,top=2.65cm,bottom=2.65cm]{geometry}
\usepackage{tikz}
\usetikzlibrary{arrows,automata,shapes,calc,patterns}

\allowdisplaybreaks

\usepackage[colorlinks=true,urlcolor=green!50!black,
citecolor=red!60!black,linkcolor=blue!75!black,linktocpage,pdfpagelabels,
bookmarksnumbered,bookmarksopen]{hyperref}

\usepackage[capitalize,nameinlink,noabbrev]{cleveref}
\newtheorem{theorem}[]{Theorem}[section]

\numberwithin{equation}{section}

\newtheorem{lemma}[theorem]{Lemma}

\newtheorem{remark}[theorem]{Remark}

\newtheorem{proposition}[theorem]{Proposition}

\addto\extrasenglish{}

\usepackage{mathtools}
\DeclarePairedDelimiter\abs{\lvert}{\rvert}
\DeclarePairedDelimiter\norm{\lVert}{\rVert}
\DeclarePairedDelimiterX{\inner}[2]{\langle}{\rangle}{#1, #2}

\newcommand{\N}{\mathbb{N}}
\newcommand{\R}{\mathbb{R}}
\newcommand{\Rn}{\mathbb{R}^3}
\newcommand{\Rb}{\mathbb{R}^3}

\newcommand{\intrn}{\int_{\Rn}}
\newcommand{\intrb}{\int_{\Rb}}
\newcommand{\Lp}[1]{\mathit{L}^{#1}(\Rn)}
\newcommand{\Hs}[1]{{\mathit H}^{#1}(\Rn)}
\newcommand{\Hsr}[1]{{\mathit H}_{r}^{#1}(\Rn)}
\newcommand{\normHs}[2]{\norm{#1}_{\Hs{#2}}}
\newcommand{\normLp}[2]{\norm{#1}_{\Lp{#2}}}

\newcommand{\diff}{\nabla}
\newcommand{\laplace}{\Delta}

\renewcommand{\(}{\left(}
\renewcommand{\)}{\right)}

\renewcommand{\[}{\left[}
\renewcommand{\]}{\right]}

\newcommand{\weakto}{\rightharpoonup}

\newcommand{\Ho}{{\mathit H}^{1} (\Rb)}

\newcommand{\tstar}{t_u^{\star}}
\newcommand{\spu}{s_u^{+}}
\newcommand{\smu}{s_u^{-}}

\newcommand{\slc}{\Lambda(c)}
\newcommand{\slp}{\Lambda^{+}(c)}
\newcommand{\slm}{\Lambda^{-}(c)}
\newcommand{\slz}{\Lambda^{0}(c)}

\newcommand{\var}{\varepsilon}

\newcommand{\wbar}{\overline{w}_{\varepsilon, t}}

\usepackage{etoolbox}

\graphicspath{{Figures/}}
\usepackage{float}
\usepackage[font=small,skip=0pt]{caption}

\begin{document}
	


\title[Multiple normalized solutions for a Sobolev critical  Schr\"odinger-Poisson-Slater equation]{Multiple normalized solutions \\ for a Sobolev critical  Schr\"odinger-Poisson-Slater equation}

\author[L. Jeanjean and T.T. Le]{Louis Jeanjean and Thanh Trung Le}

\address{
	\vspace{-0.25cm}
	\newline
	\textbf{{\small Louis Jeanjean}} 
	\newline \indent Laboratoire de Math\'{e}matiques (CNRS UMR 6623), Universit\'{e} de Bourgogne Franche-Comt\'{e}, Besan\c{c}on 25030, France}
\email{louis.jeanjean@univ-fcomte.fr}

\address{
	\vspace{-0.25cm}
	\newline
	\textbf{{\small Thanh Trung Le}} 
	\newline \indent Laboratoire de Math\'{e}matiques (CNRS UMR 6623), Universit\'{e} de Bourgogne Franche-Comt\'{e}, Besan\c{c}on 25030, France}
\email{thanh\_trung.le@univ-fcomte.fr}

\date{}
\subjclass[2010]{}
\maketitle

\begin{abstract} 
We look for solutions to the Schr\"{o}dinger-Poisson-Slater equation
\begin{align}\label{EQQ}
	- \laplace u + \lambda u - \gamma (\abs{x}^{-1} * \abs{u}^2) u  - a \abs{u}^{p-2}u = 0 \quad \text{in} \quad \Rb,
\end{align}
which satisfy 
\begin{equation*}
	\normLp{u}{2}^2 = c
\end{equation*}
for some prescribed $c>0$. Here $ u \in  \Ho$, $\gamma \in \R,$ $ a \in \R$ and $p \in (\frac{10}{3}, 6]$. When $\gamma >0$ and $a > 0$, both in the Sobolev subcritical case $p \in (\frac{10}{3}, 6)$  and in the Sobolev critical case $p=6$, we show that there exists a $c_1>0$ such that, for any $c \in (0,c_1)$, \eqref{EQQ} admits two solutions $u_c^+$ and $u_c^-$ which can be characterized respectively as a local minima  and  as  a mountain pass critical point of the associated  {\it Energy} functional restricted to the norm constraint. In the case $\gamma >0$ and $a < 0$, we show that, for any $p \in (\frac{10}{3},6]$ and any $c>0$, \eqref{EQQ} admits a solution which is a global minimizer. Finally, in the case $\gamma <0$, $a >0$ and $p=6$ we show that
\eqref{EQQ} does not admit positive solutions. 
\end{abstract}

\section{Introduction}
We consider the following Schr\"{o}dinger-Poisson-Slater equation:
\begin{align}
	i \partial_t v + \laplace v + \gamma (\abs{x}^{-1} * \abs{v}^2) v + a \abs{v}^{p-2}v = 0 \quad \text{in } \R \times \Rb, \label{eqn:SPS}
\end{align}
where $v: \R \times \Rb \to \mathbb{C}$, $\gamma \in \R,$ $ a \in \R$ and $p \in (\frac{10}{3}, 6]$. We look for standing wave solutions to \eqref{eqn:SPS}, namely to solutions of the form $v(t, x) = e^{i\lambda t}u(x)$, $\lambda \in \R$. Then the function $u(x)$ satisfies the equation
\begin{align}\label{eqn:Laplace}
	- \laplace u + \lambda u - \gamma (\abs{x}^{-1} * \abs{u}^2) u  - a \abs{u}^{p-2}u = 0 \quad \text{in } \Rb. 
\end{align}

	

Motivated by the fact that the $\mathit{L}^2 - \mbox{norm}$ is a preserved quantity of the evolution we focus on the search of solutions to \eqref{eqn:Laplace}  with prescribed $\mathit{L}^2 - \mbox{norm}$. It is standard that for some prescribed $c > 0$, a solution of \eqref{eqn:Laplace} with $\normLp{u}{2}^2 = c$ can be obtained as a critical point of the {\it Energy} functional
\begin{align*}
	F(u) := \dfrac{1}{2} \intrb \abs{\diff u}^2 dx - \dfrac{\gamma}{4} \intrb\intrb \dfrac{\abs{u(x)}^2 \abs{u(y)}^2}{\abs{x-y}} dxdy - \dfrac{a}{p} \intrb \abs{u}^p dx
\end{align*}
restricted to
\begin{align*}
	S(c) := \{u \in \Ho: \normLp{u}{2}^2 = c\}.
\end{align*}
Then the parameter $\lambda \in \R$ in \eqref{eqn:Laplace} appears as a Lagrange multiplier, it is an unknown of the problem.   

Let us define
\begin{equation}
\label{eqn:0.1}
m(c) = \inf_{u \in S(c)}F(u).
\end{equation}
Depending on the range of parameters we shall consider $m(c)$ will be finite or not. If, following the introduction of the Compactness by Concentration Principle of P. L. Lions \cite{LIONS1984-1,LIONS1984part2}, the search of normalized solutions corresponding to a global minimizer of a functional restricted to an $L^2$ norm constraint is now a classical topic, 
the search of critical points when the functional is unbounded from below on the constraint remained for a long time much less studied.  In the frame of this paper, namely for a functional corresponding to an autonomous equation lying on all the space $\R^N$, \cite{JEANJEAN1997} was for a long time the sole contribution. This direction of research was likely brought to the attention of the community by the papers 
\cite{BartschDevaleriola, BellazziniJeanjeanLuo2013} both published in 2013. Since then numerous contributions flourished within this topic and we just mention, among many possible choices, the works, \cite{BartschJeanjeanSoave16,BartschSoave2019,BellazziniJeanjean2016, Bieganowski-Mederski2020,CingolaniJeanjean2019, Gou-Zhang-2021, Lu3, Bartsch-Zhong-Zou-2021}. We also refer to  \cite{Bartsch-Molle-Rizzi-Verzini-2021} for non-autonomous problems set on $\R^N$ and to \cite{NorisTavaresVerzini2019,Pellaci-Pistoia-Vaira-Verzini-2021,Pierotti-Verzini-2017} for contributions when the underlying equation is set on a bounded domain of $\R^N$.

In the above-mentioned papers, the involved nonlinearities were Sobolev subcritical.  It was only in 2020 that was first treated in \cite{Soave2019Sobolevcriticalcase} a problem involving a Sobolev critical nonlinearity. Since then several works have explored further this direction 
\cite{AlvesJiMiyagaki2021, JeanjeanJendrejLeVisciglia2020, JeanjeanLe2020, Luo-Yang-Yang-2021, Wei-Wu2021}. \medskip

The case where $\gamma <0$ and $a>0$ in \eqref{eqn:Laplace} has been the most studied so far. When $p \in (2, \frac{10}{3})$ it can been shown that $m(c) \in (- \infty, 0]$ for any $c>0$ and it is also the case when $p = \frac{10}{3}$ and $c>0$ is small. It is shown in \cite{BellazziniSiciliano2011} that minimizer exists if $p \in (2,3)$ and $c>0$ is small enough, see also \cite{Sanchez-Soler-2004} for the special case $p=\frac{8}{3}$. The case $p \in (3, \frac{10}{3})$ was considered in \cite{BellazziniSiciliano2011-scaling, ZAMP2013}, see also \cite{Kikuchi-2007} for a closely related problem. In \cite{ZAMP2013} the existence of a threshold value $c_0>0$ such that $m(c)$ has a minimizer if and only if $c \in [c_0, \infty)$ was established. It was also proved in \cite{ZAMP2013} that a minimizer does not exist for any $c>0$ if $p=3$ or $p = \frac{10}{3}$. We also refer to \cite{Catto-Dolbeault-Sanchez-Soler-2013} for related results. When $p \in (\frac{10}{3}, 6]$ a scaling argument reveals that $m(c) = - \infty$ but nevertheless it was proved in \cite{BellazziniJeanjeanLuo2013} that, when $p \in (\frac{10}{3}, 6)$ there exists, for $c>0$ small enough a critical point of $F$ constrained to $S(c)$ at a strictly positive level. In this work we complement the result of \cite{BellazziniJeanjeanLuo2013} by showing that when $p=6$ and for any $c>0$ there does not exist positive solutions, see \cref{theorem:non-existence}. \medskip

Even if some open problems remain when $\gamma <0$ and $a>0$, we shall mainly concentrate here on the others cases: $(\gamma < 0,$ $ a < 0)$, $(\gamma > 0,$ $ a > 0)$ and $(\gamma > 0,$ $ a < 0)$. We define, for short, the following quantities
\begin{align*}
A(u): = \intrb \abs{\diff u}^2 dx,\quad B(u):= \intrb\intrb \dfrac{\abs{u(x)}^2 \abs{u(y)}^2}{\abs{x-y}} dxdy,\quad C(u):= \intrb \abs{u}^p dx.
\end{align*}
For $u \in S(c)$, we set $u^t(x) := t^{\frac{3}{2}} u(tx), t>0$, then
\begin{align*}
u^t \in S(c), \quad A(u^t) = t^2A(u), \quad B(u^t) = tB(u),\quad C(u^t) = t^{\sigma} C(u),
\end{align*}
where
\begin{align}
2 < \sigma := \dfrac{3(p-2)}{2} \leq 6, \label{eqn:11}
\end{align}
due to $p \in (\frac{10}{3}, 6]$. For $u \in S(c)$, we define the fiber map
\begin{align*}
	t \in (0, \infty) \mapsto g_u(t):= F(u^t) = \dfrac{1}{2}t^2 A(u) - \dfrac{\gamma}{4} t B(u) - \dfrac{a}{p} t^{\sigma} C(u).
\end{align*}
Hence, we have
\begin{align*}
	g_u'(t) = t A(u) - \dfrac{\gamma}{4} B(u) - \dfrac{a \sigma}{p} t^{\sigma - 1} C(u) = \dfrac{1}{t} Q(u^t),
\end{align*}
where
\begin{align*}
	Q(u) = A(u) - \dfrac{\gamma}{4} B(u) - \dfrac{a \sigma}{p} C(u).
\end{align*}
Actually the condition $Q(u) = 0$ corresponds to a Pohozaev identity and the set
\begin{align*}
	\slc := \{u \in S(c): Q(u) = 0\} = \{u \in S(c): g_u'(1) = 0\}
\end{align*}
appears as a natural constraint. Indeed, if $u \in S(c)$, then $t > 0$ is a critical point for $g_u$ if and only if $u^t \in \slc$. In particular, $u \in \slc$ if and only if $1$ is a critical point of $g_u$.

First we briefly consider the case $\gamma < 0, a < 0$. For any $u \in S(c)$, we have that $g_u'(t) > 0$ for all $t > 0$, hence the fiber map $g_u(t)$ is strictly increasing and so we can state the following non-existence result:
\begin{theorem}
	Assume that $\gamma < 0, a < 0$. Then $F(u)$ has no critical point on $S(c)$.
\end{theorem}

Next, we consider the case $\gamma > 0, a > 0$. In this case, let 
\begin{align}
	c_1 := \(\dfrac{4}{\gamma K_{H}} \dfrac{\sigma - 2}{\sigma - 1}\)^{\frac{3p-10}{4(p-3)}} \(\dfrac{p}{a \sigma (\sigma - 1) K_{GN}}\)^{\frac{1}{2(p-3)}} > 0, \label{eqn:assumption}
\end{align}
where $\sigma$ is defined by \eqref{eqn:11} and $K_{H}, K_{GN}$ are defined in \cref{lemma:1}. We also introduce the decomposition of $\slc$ into the disjoint union $\slc = \slp \cup \slz \cup \slm$, where 
\begin{align*}
\slp &:= \{u \in \slc: g_u''(1) > 0\} = \{u \in S(c): g_u'(1) = 0, g_u''(1) > 0\}, \\
\slz &:= \{u \in \slc: g_u''(1) = 0\} = \{u \in S(c): g_u'(1) = 0, g_u''(1) = 0\}, \\
\slm &:= \{u \in \slc: g_u''(1) < 0\} = \{u \in S(c): g_u'(1) = 0, g_u''(1) < 0\}.
\end{align*}
By \cref{lemma:4} and \cref{lemma:6}, for any $c \in (0, c_1)$ we have that $\slz = \emptyset$ and $\slp \neq \emptyset$, $\slm \neq \emptyset$. Since $F$ is bounded from below on $\slc$ due to \cref{lemma:5}, we can define
\begin{align}
	\gamma^{+}(c) := \inf_{u \in \slp}F(u) \quad \text{ and } \quad \gamma^{-}(c) := \inf_{u \in \slm}F(u). \label{eqn:gamma}
\end{align}
Our first main result is
\begin{theorem} \label{theorem:1}
	Let $p \in (\frac{10}{3}, 6]$. Assume that $\gamma > 0,$ $ a > 0$ and let $c_1 > 0$ be defined by \eqref{eqn:assumption}. For any $c \in (0, c_1)$, there exist $u_c^{+} \in \slp$ such that $F(u_c^{+}) = \gamma^{+}(c)$ and $u_c^{-} \in \slm$ such that $F(u_c^{-}) = \gamma^{-}(c)$. The functions $u_c^+, u_c^-$ are bounded continuous positive Schwarz symmetric functions. In addition there exist $\lambda_c^{+} > 0$ and $\lambda_c^{-} > 0$ such that $(u_c^{+}, \lambda_c^{+})$ and $(u_c^{-}, \lambda_c^{-})$ are solutions to \eqref{eqn:Laplace}.
\end{theorem}

\begin{remark}
	In \cref{theorem:1}, borrowing an approach first introduced in \cite{CingolaniJeanjean2019}, an effort is made to optimize the limit value $c_1 >0$. As a consequence,  compared to the works \cite{JeanjeanJendrejLeVisciglia2020, JeanjeanLe2020,Soave2019,Soave2019Sobolevcriticalcase} we do not benefit from the property that $\gamma^-(c) \geq 0 =  \sup_{u \in\Lambda^+(c)}F(u)$. Such property is a help to show the convergence of the Palais-Smale sequences in these works. Also, the fact that we may have $\gamma^-(c) < 0$ makes somehow more involved to prove that  the level $\gamma^-(c)$ is reached by a radially symmetric function, a Schwartz function actually, see \cref{lemma:12}. It is not clear to us if $c_1 >0$ is {\it optimal}. Nevertheless, we conjecture that there exists a $c_0 \geq c_1 >0$ such that one solution exists when $c=c_0$ and that, at least positive solutions, do not exist when $c >c_0$. 
\end{remark}

	
	\begin{remark}\label{ground-state}
	As we shall see $\gamma^+(c) < \gamma^-(c)$ and combined with the property that any critical point lies in $\Lambda(c)$ it implies that the solution $u_c^+$ obtained in \cref{theorem:1} is a ground state. Following \cite{BellazziniJeanjean2016} a ground state is defined as a solution $v \in S(c)$ to \eqref{eqn:Laplace} which has minimal Energy among all the solutions which belong to $S(c)$. Namely, if
	$$\quad F(v) =  \displaystyle \inf \big\{F(u), u \in S(c), \big(F\big|_{S(c)}\big)'(u) = 0 \big\}.$$
	\end{remark}

	If the geometrical structure of $F$ restricted to $S(c)$ is identical in the Sobolev subcritical case $ p \in (\frac{10}{3}, 6)$ and in the Sobolev critical case $p=6$, the proof that the levels $\gamma^+(c)$ and $\gamma^-(c)$ are indeed reached requires additional, more involved, arguments in the case $p=6.$ In particular, showing that $\gamma^-(c)$ is attained requires to check that the following inequality holds
	\begin{align}\label{strict-intro}
		\gamma^-(c) < \gamma^+(c)  + \dfrac{1}{3 \sqrt{a K_{GN}}}.
	\end{align}
	It is known since the pioneering work of Brezis-Nirenberg \cite{BrezisNirenberg1983} that the way to derive such a strict upper bound is through the use of testing functions. In \cite{JeanjeanLe2020}, considering the equation 
	\begin{align}
-\laplace u - \lambda u - \mu \abs{u}^{q-2} u - \abs{u}^{2^*-2} u = 0 \quad \mbox{in } \R^N, \label{eqn:LaplaceL}
\end{align}
with $N \geq 3,$ $ \mu >0$, $2 < q < 2 + \frac{4}{N}$ and $2^* = \frac{2N}{N-2}$ we face the need to establish a similar inequality. We constructed test functions which could be viewed as the sum of a truncated extremal function of the Sobolev inequality on $\R^N$ centered at the origin and of $u_c^+$ translated far away from the origin. This choice of testing functions was sufficient to prove our strict inequality when $N \geq 4$ but we missed it in the case $N=3$. Note that the approach developed in  \cite{JeanjeanLe2020} proved nevertheless adequate to deal with the equation
	\begin{equation*}
	\sqrt{- \Delta} u = \lambda u + \mu |u|^{q-2}u + |u|^{2^* -2}u, \quad u \in \mathit{H}^{1/2}(\R^N),
	\end{equation*}
	with $N \geq 2,$ $ q \in (2, 2 + \frac{2}{N})$, $ 2^* = \frac{2N}{N-1}$, that was studied in \cite{Luo-Yang-Yang-2021}. Very recently, in \cite{Wei-Wu2021} the authors introduced an alternative choice of testing functions which allowed to treat, in a unified way, the case $N=3$ and $N \geq 4$ for \eqref{eqn:LaplaceL}. The strategy in \cite{Wei-Wu2021}, recording of the one introduced by G. Tarantello in \cite{Tarantello92}, is on the contrary, to located the extremal functions where the solution $u_c^+$ takes its greater values (the origin thus). The idea behind the proof is that the interaction decreases the value of the {\it Energy}  with respect to the case where the supports would be disjoint. In this paper, where \eqref{eqn:Laplace} is set on $\R^3$, we believe in view of our experience on \eqref{eqn:LaplaceL}, more appropriate to follow the approach of \cite{Wei-Wu2021} to check the inequality \eqref{strict-intro} for any $c \in (0,c_1)$. \smallskip 
	
	The results of \cref{theorem:1} are complemented in several directions. First, we show that the solution $u^+(c)$ obtained in \cref{theorem:1} can be characterized as a local minima for $F$ restricted to $S(c)$. We treat here the full range $p \in (\frac{10}{3}, 6]$ with a single proof. More precisely we show,
	\begin{theorem}\label{th-minimization}
	Let $p \in (\frac{10}{3}, 6]$. Assume that $\gamma > 0,$ $ a > 0$ and let $c \in (0,c_1)$. Then
	we have $\slp \subset V(c)$ and
	\begin{align*} 
		\gamma^+(c) = \inf_{u \in \slp}F(u) = \inf_{u \in V(c)}F(u)
	\end{align*}
	where \begin{align*}
		V(c) := \{u \in S(c) | A(u) < k_1\}
	\end{align*}
	for some $k_1 >0$ independent of $c \in (0,c_1)$ (see \eqref{eq-notations} for the definition of $k_1>0$). In addition, any minimizing sequence for $F$ on $V(c)$ is, up to translation, strongly convergent in $\Hs{1}$. 
	\end{theorem}

	\begin{remark}
	The proof of \cref{lemma:27} which is a key step to established \cref{th-minimization}, reveals some additional properties of the set $V(c)$. Indeed, we have that
	$V(c) \subset S(c) \backslash \Lambda^-(c)$ and thus $V(c)$ is {\it separating} the sets $\Lambda^+(c)$ and $\Lambda^-(c)$. Also, for any $0 < c, \tilde{c} < c_1$, we have that $A(u) < k_1 < A(v)$ for all $u \in \Lambda^+(c), v \in \Lambda^-(\tilde{c})$, see \eqref{eqn:43} and \eqref{eqn:41}. 
	\end{remark}

	\begin{remark} 
	To prove that the minimizing sequences for $F$ on $V(c)$ are, up to translation, strongly convergent in $\Hs{1}$ we follow an approach due to \cite{Ikoma2014} that has already been used several times, see, for example, \cite{GouJeanjean2016, JeanjeanJendrejLeVisciglia2020, Luo-Yang-Yang-2021}. The first step in this approach is to show that the sequences do not vanish. When $p=6$, we rely for this, in an essential way, on the fact that $c_1>0$ is sufficiently small, see 
	\cref{lemma:30}. This fact is also used to end the proof. Finally, note that since we allow the possibility that $\inf_{u \in \partial V(c)}F(u) < 0$ where 
	$\partial V(c) := \{u \in S(c) | A(u) = k_1\}$ we must check that the minimizers do ly in $V(c)$.
	\end{remark}

	Let us now denote 
	\begin{align*} 
\mathcal{M}_c := \{u \in V(c) : F(u) = \gamma^+(c)\}.
\end{align*}
In view of \cref{ground-state}, $\mathcal{M}_c $ is the set of all ground states.
The property that any minimizing sequence for $F$ restricted to $V(c)$ is, up to translation, strongly converging is known to be a key ingredient to show that the set $\mathcal{M}_c$ is orbitally stable. If $p \in (\frac{10}{3}, 6)$ the orbital stability of $\mathcal{M}_c$ indeed follows directly from \cref{th-minimization} by the classical arguments of \cite{CazenaveLions1982}. In the case $p=6$ the situation is more delicate as the existence of a uniform $H^1(\R^3)$ bound on the solution of \eqref{eqn:SPS} during its lifespan is not sufficient to  guarantee that blow-up may not occurs. We refer to \cite{Cazenave2003semilinear} for more details. We do not prove anything in that direction but strongly believe that the set $\mathcal{M}_c$ is orbitally stable. Actually, such a result has been obtained on the equation \eqref{eqn:LaplaceL} in \cite{JeanjeanJendrejLeVisciglia2020}. \smallskip

	We also discuss the behavior of the associated Lagrange multipliers and show that if the behavior of $\lambda^+_c$ is essentially the same for the cases $p \in (\frac{10}{3}, 6)$ and $p=6$, see \cref{lemma:22}, there is a distinct behavior for $\lambda_c^-$, see Lemmas \ref{lemma:23t} and \ref{lemma:25}. In particular, \cref{lemma:25} suggests that there may exist two distinct positive  solutions to \eqref{eqn:Laplace} for any fixed $\lambda >0$ sufficiently small. Finally, in \cref{lemma:10}, we establish the property that the map $c \mapsto \gamma^-(c)$ is strictly decreasing. \medskip

Next, we consider the case  $\gamma > 0,$ $ a < 0$. Recalling the definition of $m(c)$ given in \eqref{eqn:0.1} we show in \cref{lemma:5.1}, that $-\infty < m(c) < 0$ and then we prove the following result.
\begin{theorem} \label{theorem:2}
Let $p \in (\frac{10}{3},6]$, $\gamma > 0$ and $ a < 0$. For any $c >0$, the infimum $m(c)$ is achieved and any minimizing sequence for \eqref{eqn:0.1} is, up to translation, strongly convergent in $\Ho$ to a solution of \eqref{eqn:Laplace}. In addition, the associated Lagrange multiplier is positive.
\end{theorem}

Even if the proof of \cref{theorem:2} follows the lines of the proof of \cref{th-minimization}, the change of sign in front of the power term requires some adaptations, see  \cref{lemma:5.5} and \cref{lemma:5.4}.
 Here again the orbital stability of the set of minimizers should follow directly from the classical arguments of \cite{CazenaveLions1982} 
if  $p \in (\frac{10}{3}, 6)$ and it should also be the case when $p=6$ by adapting the arguments of \cite{JeanjeanJendrejLeVisciglia2020}. Note that we also study the behavior of the associated Lagrange multipliers in \cref{lemma:5.6}. \medskip

In the last part of the paper we consider the case $\gamma <0$, $a>0$ and $p=6$. 
\begin{theorem} \label{theorem:non-existence}
	Let $p=6$, $\gamma <0$ and $a>0$. For any $c > 0$, we have that
	\begin{enumerate}[label=(\roman*), ref = \roman*]
		\item\label{point:theorem:non:i} If $u \in \Hs{1}$ is a non-trivial solution to \eqref{eqn:Laplace} then the associated Lagrange multiplier $\lambda$ is negative and 
		\begin{align*}
			F(u) > \dfrac{1}{3\sqrt{aK_{GN}}}.
		\end{align*}
	
		\item\label{point:theorem:non:ii} \cref{eqn:Laplace} has no positive solution in  $ \Hs{1}$.
	\end{enumerate}
\end{theorem}

\begin{remark} 
Under the assumptions of \cref{theorem:non-existence},  it is possible to prove that $$ \inf_{u \in \Lambda(c)}F(u) = \dfrac{1}{3\sqrt{aK_{GN}}}.$$
\end{remark}

\begin{remark} 
	In \cite[Theorem 1.2]{Soave2019Sobolevcriticalcase}, considering the equation
	\begin{align}
		-\laplace u - \lambda u - \mu \abs{u}^{q-2} u - \abs{u}^{2^*-2} u = 0 \quad \mbox{in } \R^N, \label{eqn:LaplaceS2}
	\end{align}
	with $N \geq 3$, $2<q<2^*$ and $\mu < 0$, it was proved that \eqref{eqn:LaplaceS2} has no positive solution $u\in \mathit{H}^1(\R^N)$ if $N=3,4$ or if $N \geq 5$ under the additional assumption
	$u \in \mathit{L}^p(\R^N)$ for some $p \in \(0, \frac{N}{N-2}\]$. In \cref{remrak:6.1}, partly using arguments used in the proof of \cref{theorem:non-existence},
	we improve  \cite[Theorem 1.2]{Soave2019Sobolevcriticalcase} showing that \eqref{eqn:LaplaceS2} has no positive solution in $\mathit{H}^1(\R^N)$ for all $N \geq 3$ and no non-trivial radial solution for $N \geq 3$ and $q > 2 + \frac{2}{N-1}$.
\end{remark}

\begin{remark} 
We propose as an open problem to investigate if there are radial solutions under the assumptions of \cref{theorem:non-existence}. See \cref{open-problem}  in that direction.
\end{remark}

The paper is organized as follows. In \cref{Section2} we recall some classical inequalities and present some preliminary results. \cref{Section3} is devoted to the treatment of the case $\gamma >0$, $a>0$ and $p \in (\frac{10}{3},6]$. In \autoref{Subsection3-1} we make explicit the geometrical structure of $F$ on $S(c)$ and show the existence of a bounded Palais-Smale sequence $(u_n^+) \subset \Lambda^+(c)$ at the level $\gamma^+(c)$ and of a bounded Palais-Smale sequence $(u_n^-) \subset \Lambda^-(c)$ at the level $\gamma^-(c)$. In \autoref{Subsection3-2} we give the proof of \cref{theorem:1} in the Sobolev subcritical case.  \autoref{Subsection3-3} is devoted to the proof of \cref{theorem:1} in the critical case. In \autoref{Subsection3-4} we prove the convergence of all minimizing sequences associated to $\gamma^+(c)$, namely \cref{th-minimization}. 
The behavior of the Lagrange multipliers and the property of the map $c \mapsto \gamma^{-}(c)$ are studied in \autoref{Subsection3-5} and \autoref{Subsection3-6}, respectively.
In \cref{Section4} we treat the case $\gamma >0$, $a<0$ and $p \in (\frac{10}{3}, 6]$ and we prove \cref{theorem:2}.
Finally, in \cref{Section:6}, we consider the case $\gamma <0$, $a >0$ and $p=6$, and prove \cref{theorem:non-existence}.\medskip

{\bf Notation:} For $p \geq 1$, the $\mathit{L}^p$-norm of $u \in \Hs{1}$ is denoted by $\normLp{u}{p}$. We denote by $\Hsr{1}$ the subspace of functions in $\Ho$ which are radially symmetric with respect to $0$. The notation $a \sim b$ means that $Cb \leq a \leq C'b$ for some $C, C' > 0$. The open ball in $\Rn$ with center at $0$ and radius $R > 0$ is denoted by $B_R$.

{\bf Addendum :} After the completion of this paper, we were informed of the work \cite{Yao-Sun-Wu-2021} in which the authors consider a general class of problems which, when $p \in (\frac{10}{3}, 6)$,  covers \eqref{eqn:Laplace} as a special case.  There are thus some partial overlap, in the Sobolev subcritical case, between \cite[Theorem 1.3 (a) (ii)]{Yao-Sun-Wu-2021} and \cref{theorem:1} and between \cite[Theorem 1.6 (a) (iv)]{Yao-Sun-Wu-2021} and \cref{theorem:2}.  However the scope of the two works is widely distinct.

	\section{Preliminary results}\label{Section2}
In this section we present various preliminary results. When it is not specified they are assumed to hold for $\gamma \in \R,$ $ a\in\R,$ $ p \in \(\frac{10}{3}, 6\]$ and any $c > 0$.
Firstly, we present the definitions of $\slc,$ $ \slm,$ $ \slz,$ $ \slm$ via $A(u),$ $ B(u)$ and $C(u)$:
\begin{align*}
\slc &=\Bigg\{u \in S(c): A(u) = \dfrac{\gamma}{4} B(u) + \dfrac{a \sigma}{p} C(u) \Bigg\}, \\
\slp &=\Bigg\{u \in S(c): A(u) = \dfrac{\gamma}{4} B(u) + \dfrac{a \sigma}{p} C(u), A(u) > \dfrac{a \sigma (\sigma - 1)}{p} C(u) \Bigg\}, \\ 
\slz &=\Bigg\{u \in S(c): A(u) = \dfrac{\gamma}{4} B(u) + \dfrac{a \sigma}{p} C(u), A(u) = \dfrac{a \sigma (\sigma - 1)}{p} C(u) \Bigg\}, \\
\slm &=\Bigg\{u \in S(c): A(u) = \dfrac{\gamma}{4} B(u) + \dfrac{a \sigma}{p} C(u), A(u) < \dfrac{a \sigma (\sigma - 1)}{p} C(u) \Bigg\}.
\end{align*}

\begin{lemma} \label{lemma:1}
	Let $u \in S(c)$, there exists
	\begin{enumerate}[label=(\roman*), ref = \roman*]
		\item\label{point:1i} a constant $K_{H} > 0$ such that $B(u) \leq K_{H}  \sqrt{A(u)} c^{\frac{3}{2}}$.
		\item\label{point:1ii} a constant $K_{GN} > 0$ such that $C(u) \leq  K_{GN} [A(u)]^{\frac{\sigma}{2}} c^{\frac{6-p}{4}}$.
	\end{enumerate}
\end{lemma}
\begin{proof}
	We first recall the Hardy-Littlewood-Sobolev inequality (see \cite[Chapter 4]{LiebLoss2001analysis}):
	\begin{align}\label{ineq:HLS}
	\abs*{\int_{\R^N} \int_{\R^N} \dfrac{f(x)g(y)}{\abs{x-y}^{\lambda}} dxdy} \leq C(N,\lambda, p, q) \norm{f}_{\mathit{L}^p(\R^N)} \norm{g}_{\mathit{L}^q(\R^N)},
	\end{align}
	where $f \in \mathit{L}^p(\R^N)$, $g \in \mathit{L}^q(\R^N)$, $p, q > 1$, $0 < \lambda < N$ and
	\begin{align*}
	\dfrac{1}{p} + \dfrac{1}{q} + \dfrac{\lambda}{N} = 2.
	\end{align*}
	Let us also recall the Gagliardo-Nirenberg inequality (see \cite{Nirenberg1985}) and the Sobolev inequality (see \cite[Theorem 9.9]{Brezis-2011}) in the unified form: if $N \geq 3$ and $p \in [2, \frac{2N}{N-2}]$ then
	\begin{align*}
	\norm{f}_{\mathit{L}^{p}(\R^N)} \leq C(N, p) \norm{\diff f}_{\mathit{L}^2(\R^N)}^{\beta} \norm{f}_{\mathit{L}^2(\R^N)}^{(1-\beta)}, \qquad \text{with } \beta = N\(\dfrac{1}{2} - \dfrac{1}{p}\).
	\end{align*}
	 Applying the Hardy-Littlewood-Sobolev inequality we obtain
	\begin{align}
	B(u) = \intrb\intrb \dfrac{\abs{u(x)}^2 \abs{u(y)}^2}{\abs{x-y}} dxdy \leq K_1 \norm{u}_{\Lp{\frac{12}{5}}}^4 \label{eqn:2.1}
	\end{align}
	and thus using the Gagliardo-Nirenberg inequality, we get
	\begin{align*}
	B(u) \leq K_1 \norm{u}_{\Lp{\frac{12}{5}}}^4 \leq K_1 K_2 \normLp{\diff u}{2} \normLp{u}{2}^3 = K_{H} \sqrt{A(u)} c^{\frac{3}{2}}.
	\end{align*}
	Finally, applying the Sobolev, Gagliardo-Nirenberg inequality, we have
	\begin{align*}
	C(u) = \normLp{u}{p}^p \leq K_{GN} \normLp{\diff u}{2}^{\sigma} \normLp{u}{2}^{\frac{6-p}{2}} = K_{GN} [A(u)]^{\frac{\sigma}{2}} c^{\frac{6-p}{4}}.
	\end{align*}
\end{proof}

\begin{lemma}\label{lemma:2.3}
	Let $p \in (\frac{10}{3}, 6]$. Assume that $\gamma \in \R$ and $a \in \R$. If $u \in \Hs{1}$ is a weak solution to
	\begin{align}
	- \laplace u + \lambda u - \gamma (\abs{x}^{-1} * \abs{u}^2) u  - a \abs{u}^{p-2}u = 0, \label{eqn:6}
	\end{align}
	then $Q(u) = 0$. Moreover, if $u \neq 0$ then we have 
	\begin{enumerate}[label=(\roman*), ref = \roman*]
		\item\label{point:2.3:i}  $\lambda > 0$ if $\gamma > 0$ and $p \in (\frac{10}{3}, 6]$,
		\item\label{point:2.3:ii} $\lambda < 0$ if $\gamma < 0$ and $p = 6$.
	\end{enumerate}
\end{lemma}
\begin{proof}
	Our proof is inspired by \cite[Lemma 4.2]{BellazziniJeanjeanLuo2013}. The following Pohozaev type identity holds for 
	$u \in \Ho$ weak solution of \eqref{eqn:6} (\cite{AprileMugnai2004_Non}, also see \cite[Theorem 2.2]{DavidRuiz2006}),
	\begin{align}
	\dfrac{1}{2} A(u) + \dfrac{3 \lambda}{2} D(u) - \dfrac{5 \gamma}{4} B(u) - \dfrac{3a}{p} C(u) = 0, \quad \mbox{where} \quad D(u) = \normLp{u}{2}^2. \label{eqn:7}
	\end{align}
	By multiplying \eqref{eqn:6} by $u$ and integrating, we derive a second identity
	\begin{align}
	A(u) + \lambda_c D(u) - \gamma B(u) - a C(u) = 0. \label{eqn:8}
	\end{align}
	Combining \eqref{eqn:7} and \eqref{eqn:8}, we get
	\begin{align*}
	A(u) - \dfrac{\gamma}{4} B(u) - \dfrac{a \sigma}{p} C(u) = 0.
	\end{align*}
	This means that $Q(u) = 0$. Using \eqref{eqn:7} and \eqref{eqn:8} again, we obtain
	\begin{align}
		2(6-p) A(u) + (5p -12)\gamma B(u) = 2(3p-6)\lambda D(u). \label{eqn:lambda}
	\end{align}
	If $\gamma > 0$ and $p \in (\frac{10}{3}, 6]$, we have
	\begin{align*}
	2(6-p) \geq 0, \qquad (5p -12)\gamma > 0, \qquad 2(3p-6) > 0.
	\end{align*}
	Hence, $\lambda > 0$. If $\gamma < 0$ and $p = 6$, we have
	\begin{align*}
		2(6-p) = 0, \qquad (5p -12)\gamma = 18 \gamma < 0, \qquad 2(3p-6) = 24 > 0.
	\end{align*}
	This implies that $\lambda < 0$.
\end{proof}

\begin{lemma}\label{lemma:bounded-positive-solution}
	Let $p \in (\frac{10}{3}, 6]$. Assume that $\gamma \in \R$ and $a \in \R$. If $u \in \Hs{1}$ is a weak solution to
	\begin{align}
		- \laplace u + \lambda u - \gamma (\abs{x}^{-1} * \abs{u}^2) u  - a \abs{u}^{p-2}u = 0,
	\end{align}
	then $u \in \Lp{\infty} \cap \mathit{C}(\Rn)$. Moreover, in case $\gamma > 0,$ $ a > 0$ we have that if $u \not\equiv 0 $ and $u \geq 0$ then $u > 0$.
\end{lemma}
\begin{proof}
	Applying \cite[Theorem 2.1]{Li-Ma-2020}, we get that $u \in \mathit{W}_{loc}^{2,r}(\Rn)$ for every $r > 1$ and hence $u \in \mathit{C}(\Rn)$. Since $u \in \Hs{1}$, the  Sobolev embedding (see \cite[Corollary 9.10]{Brezis-2011}) implies that $|u|^2 \in \Lp{q}$ for every $q \in [1,3]$. 
	Now,  setting $K := |x|^{-1}$, we write $K := K_1 + K_2$ where $K_1:= K$ on $B(0,1)$, $K_1:=0$ on $\R^3 \backslash \ B(0,1)$ and $K_2:= K - K_1$. 
	Clearly $K_1 \in \Lp{2}$ and $K_2 \in \Lp{4}$. Applying \cite[Lemma 2.20]{LiebLoss2001analysis} with $K_1 \in \Lp{2}$, $|u|^2 \in \Lp{2}$ and with $K_2 \in \Lp{4}$, $|u|^2 \in \Lp{\frac{4}{3}}$, we obtain that $K_1*|u|^2$ and $K_2*|u|^2$ are continuous. Also
	\begin{align*}
		\lim_{|x| \to \infty} (K_1*|u|^2) (x) = 0 \quad\mbox{and}\quad \lim_{|x| \to \infty} (K_2*|u|^2) (x) = 0.
	\end{align*}
	Hence, we get that $K*|u|^2$ is continuous and
	\begin{align} \label{eqn:tend-to-0}
		\lim_{|x| \to \infty} (K*|u|^2) (x) = 0.
	\end{align}
	Therefore, $K*|u|^2$ is bounded. At this point, we deduce from \cite[Proposition B.1]{Soave2019Sobolevcriticalcase} that $u \in \Lp{\infty}$.	

	Now, if we assume that $\gamma > 0,$ $ a > 0$, $u \not\equiv 0$, $u \geq 0$, setting $v := -u \leq 0$ we get
	\begin{align*}
		- \laplace v + \lambda v =  \gamma (\abs{x}^{-1} * \abs{v}^2) v  + a \abs{v}^{p-2}v \leq 0.
	\end{align*}
	By \cref{lemma:2.3}, we have that $\lambda > 0$. We assume that there exists $x_0 \in \Rn$ such that $v(x_0) = 0$. For all $R > |x_0|$, we have that $v \in \mathit{W}^{2,r}(B_R)$ for every $r > 1$, $Lv := - \laplace v + \lambda v \leq 0$ in $B_R$ with $\lambda > 0$ and $M:= \max_{x \in B_R} v = 0$. At this point, applying \cite[Theorem 3.27]{Troianiello-1987}, in the particular case where $\Gamma = \emptyset$, we obtain that $v \equiv 0$ in $B_R$, and hence $u \equiv 0$ in $B_R$. The value $R>0$ being arbitrarily large, this contradicts our assumption that $u \not\equiv 0$ and we conclude that $u > 0$.
\end{proof}


Following \cite{BerestyckiLions1983_2}, we recall that, for any $c > 0$, $S(c)$ is a submanifold codimension 1 of $\Ho$ and the tangent space at a point $u \in S(c)$ is defined as
\begin{align*}
T_uS(c) = \{\varphi \in \Ho: \inner{u}{\varphi}_{\Lp{2}} = 0\}.
\end{align*}
The restriction ${F}_{|S(c)} :S(c) \to \R$ is a $\mathit{C}^1$ functional on $S(c)$ and for any $u \in S(c)$ and any $v \in T_uS(c)$, we have
\begin{align*}
\inner{{F}'_{|S(c)}}{\varphi} = \inner{F'}{\varphi}.
\end{align*}
We use the notation $\norm{{dF}_{|S(c)}}_{*}$ to indicate the norm in the cotangent space $T_uS(c)'$, i.e the dual norm induced by the norm of $T_uS(c)$, i.e
\begin{align}
	\norm{{dF}_{|S(c)} (u)}_{*} := \sup_{\norm{\varphi} \leq 1, \varphi \in T_uS(c)} \abs{dF(u)[\varphi]}. \label{eqn:21}
\end{align}
We recall the following result, see \cite[Lemma 3.1]{JeanjeanLe2020},
\begin{lemma} \label{lemma:2.4}
	For $u \in S(c)$ and $t > 0$, the map
	\begin{align*}
	T_uS(c) \to T_{u^t} S(c), \quad \psi \mapsto \psi^t
	\end{align*}
	is a linear isomorphism with inverse $$T_{u^t}S(c) \to T_{u} S(c), \quad \phi \mapsto \phi^{\frac{1}{t}}.$$
\end{lemma}

Next, we recall a result concerning the convergence of the term $B$, see {\cite[Lemma 2.1]{DavidRuiz2006}},
\begin{lemma} \label{lemma:2.5}
	Let $(u_n)$ be a sequence satisfying $u_n \weakto u$ weakly in $\Hsr{1}$. Then we have $B(u_n) \to B(u)$.
\end{lemma}



\section{The case $\gamma > 0,$ $ a > 0$ and $ p \in (\frac{10}{3},  6]$.} \label{Section3}

\subsection{The geometrical structure and the existence of bounded Palais-Smale sequences for $ p \in (\frac{10}{3},  6]$}\label{Subsection3-1}
In this subsection, we follow the approach first introduced in \cite{CingolaniJeanjean2019}. We shall always assume that $\gamma > 0,$ $ a > 0$ and $ p \in (\frac{10}{3},  6]$.
\begin{lemma} \label{lemma:5}
For any $c \in (0, \infty)$, 
	$F$ restricted to $\slc$ is coercive on $\Ho$, namely when $(u_n) \subset H^1(\R^3)$ satisfies $||u_n|| \to + \infty$ then $F(u_n) \to + \infty$. In particular $F$ restricted to $\slc$ is bounded from below. 
\end{lemma}
\begin{proof}
	Let $u \in \slc$. Taking into account that
	\begin{align*}
	\dfrac{a}{p} C(u) = \dfrac{1}{\sigma} A(u) - \dfrac{\gamma}{4\sigma} B(u),
	\end{align*}
	and using \cref{lemma:1}\eqref{point:1i}, we obtain
	\begin{align}
	\begin{split}
		F(u) &= \dfrac{1}{2} A(u) - \dfrac{\gamma}{4} B(u) - \dfrac{a}{p} C(u)
		= \dfrac{1}{2} A(u) - \dfrac{\gamma}{4} B(u) - \dfrac{1}{\sigma} A(u) + \dfrac{\gamma}{4\sigma} B(u) \\
		&= \dfrac{\sigma - 2}{2 \sigma} A(u) - \dfrac{\gamma(\sigma - 1)}{4} B(u) 
		\geq \dfrac{\sigma - 2}{2 \sigma} A(u) - \dfrac{\gamma(\sigma - 1)}{4} K_{H}  \sqrt{A(u)} c^{\frac{3}{2}}.
	\end{split} \label{eqn:29}
	\end{align}
	This concludes the proof.
\end{proof}

For any $u \in S(c)$, we recall that
\begin{align*}
&g_u(t)= F(u^t) = \dfrac{1}{2}t^2 A(u) - \dfrac{\gamma}{4} t B(u) - \dfrac{a}{p} t^{\sigma} C(u),\\
&g_u'(t) = t A(u) - \dfrac{\gamma}{4} B(u) - \dfrac{a \sigma}{p} t^{\sigma - 1} C(u) = \dfrac{1}{t} Q(u^t),\\
&g_u''(t) = A(u) - \dfrac{a \sigma (\sigma - 1)}{p} t^{\sigma - 2} C(u).
\end{align*}
For any $u \in S(c)$, we set
\begin{align*}
\tstar := \(\dfrac{pA(u)}{a \sigma (\sigma - 1) C(u)}\)^{\frac{1}{\sigma - 2}}.
\end{align*}
This implies that $\tstar$ is the unique solution of equation $g_u''(t) = 0$. So, we have
\begin{align}
	g_u''(\tstar) = 0, \quad g_u''(t) > 0 \text{ if } 0<t<\tstar, \quad g_u''(t) < 0 \text{ if } t>\tstar. \label{eqn:2}
\end{align}

\begin{lemma} \label{lemma:3}
For any $c \in (0, c_1)$ and any $u \in S(c)$,  we have $g_u'(\tstar) > 0$.
\end{lemma}

\begin{proof}
	Let $u \in S(c)$ be arbitrary. By the definition of $\tstar$ and by $g_u''(\tstar) = 0$, we have	
	\begin{align*}
	g_u'(\tstar) &= \tstar A(u) - \dfrac{\gamma}{4} B(u) - \dfrac{a \sigma}{p} (\tstar)^{\sigma - 1} C(u) 
	= \tstar A(u) - \dfrac{\gamma}{4} B(u) - \dfrac{1}{\sigma - 1} \tstar A(u)\\
	&= \dfrac{\sigma - 2}{\sigma - 1} \tstar A(u) - \dfrac{\gamma}{4} B(u) 
	= \dfrac{\sigma - 2}{\sigma - 1} \(\dfrac{pA(u)}{a \sigma (\sigma - 1) C(u)}\)^{\frac{1}{\sigma - 2}} A(u) - \dfrac{\gamma}{4} B(u)\\
	&= \sqrt{A(u)} \[\dfrac{\sigma - 2}{\sigma - 1} \(\dfrac{pA(u)}{a \sigma (\sigma - 1) C(u)}\)^{\frac{1}{\sigma - 2}} \sqrt{A(u)} - \dfrac{\gamma}{4} \dfrac{B(u)}{\sqrt{A(u)}}\] \\
	&= \sqrt{A(u)} \[\dfrac{\sigma - 2}{\sigma - 1} \(\dfrac{p[A(u)]^{\frac{\sigma}{2}}}{a \sigma (\sigma - 1) C(u)}\)^{\frac{1}{\sigma - 2}} - \dfrac{\gamma}{4} \dfrac{B(u)}{\sqrt{A(u)}}\].	
	\end{align*}
	Applying \cref{lemma:1}, we obtain
	\begin{align*}
	g_u'(\tstar) &\geq \sqrt{A(u)} \[\dfrac{\sigma - 2}{\sigma - 1} \(\dfrac{p[A(u)]^{\frac{\sigma}{2}}}{a \sigma (\sigma - 1) K_{GN} [A(u)]^{\frac{\sigma}{2}} c^{\frac{6-p}{4}}}\)^{\frac{1}{\sigma - 2}} - \dfrac{\gamma}{4} \dfrac{K_{H}  \sqrt{A(u)} c^{\frac{3}{2}}}{\sqrt{A(u)}}\] \\
	&=\sqrt{A(u)} \[\dfrac{\sigma - 2}{\sigma - 1} \(\dfrac{p}{a \sigma (\sigma - 1) K_{GN}  c^{\frac{6-p}{4}}}\)^{\frac{1}{\sigma - 2}} - \dfrac{\gamma}{4} K_{H} c^{\frac{3}{2}}\].
	\end{align*}
	By direct computations, we now have
		\begin{align*}
			\dfrac{\sigma - 2}{\sigma - 1} \(\dfrac{p}{a \sigma (\sigma - 1) K_{GN}  c^{\frac{6-p}{4}}}\)^{\frac{1}{\sigma - 2}} - \dfrac{\gamma}{4} K_{H} c^{\frac{3}{2}} > 0 \iff c < c_1.
		\end{align*}		
	Thus, we obtain that if $0 < c< c_1$ then $g_u'(\tstar) > 0$.
\end{proof}

\begin{lemma} \label{lemma:4}
	For any $c \in (0, c_1)$, it holds that $\slz = \emptyset$.
\end{lemma}
\begin{proof}
	We assume that there exists $u \in \slz$. Since $g_u''(1) = 0$ and $\tstar$ is the unique solution of equation $g_u''(t) = 0$, we have $\tstar = 1$. So, we have $g_u'(\tstar) = g_u'(1) = 0$. This contradicts with $g_u'(\tstar) > 0$ in \cref{lemma:3}. Thus, we obtain $\slz = \emptyset$.
\end{proof}

\begin{lemma} \label{lemma:6}
For any $c \in (0, c_1)$ and any $u \in S(c)$, there exists
	\begin{enumerate}[label=(\roman*)]
		\item a unique $\spu \in (0, \tstar)$ such that $\spu$ is a unique local minimum point for $g_u$ and $u^{\spu} \in \slp$.
		\item a unique $\smu \in (\tstar, \infty)$ such that $\smu$ is a unique local maximum point for $g_u$ and $u^{\smu} \in \slm$. 
	\end{enumerate}
	Moreover, the maps $u \in S(c) \mapsto \spu \in \R$ and $u \in S(c) \mapsto \smu \in \R$ are of class $\mathbb{C}^1$.
\end{lemma}
\begin{proof}
	Taking into account that
	\begin{align*}
		g_u'(t) = t A(u) - \dfrac{\gamma}{4} B(u) - \dfrac{a \sigma}{p} t^{\sigma - 1} C(u),
	\end{align*}
	we have $g_u'(t) \to - \frac{\gamma}{4} B(u) < 0$ as $t \to 0$ and $g_u'(t) \to - \infty$ as $t \to +\infty$ due to $\sigma - 1 > 1$. By \cref{lemma:3}, we have $g'_u(\tstar) > 0$. Therefore, the equation $g_u'(t) = 0$ has at least two solutions $\spu$ and $\smu$ with $0< \spu < \tstar < \smu$. By \eqref{eqn:2}, we have $g_u''(t) > 0$ for all $0<t<\tstar$. Hence, $g_u'(t)$ is strictly increasing function on $(0, \tstar)$ and consequently $\spu \in (0, \tstar)$ is the unique local minimum point for $g_u$ and $u^{\spu} \in \slp$ due to $g''_{u^{\spu}}(1) = g_u''(\spu) > 0$. By the same argument, we obtain that $\smu \in (\tstar, \infty)$ is a unique local maximum point for $g_u$ and $u^{\smu} \in \slm$.
	
	In order to prove that $u \mapsto \smu$ are of class $\mathbb{C}^1$, we follow the argument in \cite[Lemma 5.3]{Soave2019}. It is a direct application of the Implicit Function Theorem on $\mathbb{C}^1$-function $\varphi(t, u)= g_u'(t)$. Taking into account that $\varphi(\smu, u) = g_u'(\smu)=0$, $\partial_t \varphi(\smu, u) = g_u''(\smu) < 0$ and $\slz = \emptyset$, we obtain $u \mapsto \smu$ is of class $\mathbb{C}^1$. The same argument proves that $u \mapsto \spu$ is of class $\mathbb{C}^1$.
\end{proof}

\begin{lemma}\label{lemma:16}
	For any $c \in (0, c_1)$, it holds that
	\begin{enumerate}[label=(\roman*), ref=\roman*]
		\item\label{point:1} $F(u) < 0$ for all $u \in \slp$,
		\item\label{point:2} there exists $\alpha:= \alpha(c) > 0$ such that $A(u) \geq \alpha$ for all $u \in \slm$.
	\end{enumerate}
\end{lemma}
\begin{proof}
	Let $u \in \slp$, taking into account that
	\begin{align*}
	A(u) = \dfrac{\gamma}{4} B(u) + \dfrac{a \sigma}{p} C(u), \qquad A(u) > \dfrac{a \sigma (\sigma - 1)}{p} C(u),
	\end{align*}
	we obtain 
	\begin{align*}
	F(u) &= \dfrac{1}{2} A(u) - \dfrac{\gamma}{4} B(u) - \dfrac{a}{p} C(u)
	= \dfrac{1}{2} A(u) - \dfrac{\gamma}{4} B(u) - \dfrac{a \sigma}{p} C(u) + \dfrac{a (\sigma -1)}{p} C(u)\\
	&< \dfrac{1}{2} A(u) - A(u) + \dfrac{1}{\sigma} A(u) = \dfrac{2-\sigma}{2\sigma} A(u).
	\end{align*}
	Since $\sigma > 2$, we have $F(u) < 0$. The point \eqref{point:1} is proved.
	
	Let $u \in \slm$, taking into account that
	\begin{align*}
	A(u) < \dfrac{a \sigma (\sigma - 1)}{p} C(u),
	\end{align*}
	and using \cref{lemma:1}, we obtain that
	\begin{align*}
	A(u) < \dfrac{a \sigma (\sigma - 1)}{p} K_{GN} c^{\frac{6-p}{4}} [A(u)]^{\frac{\sigma}{2}}.
	\end{align*}
	Since $\sigma > 2$, the point \eqref{point:2} follows.
\end{proof}

We define
\begin{align*}
	S_r(c) := S(c) \cap \Hsr{1}, \qquad \Lambda_r(c) := \slc \cap \Hsr{1}, \qquad \Lambda^{\pm}_r(c) := \Lambda^{\pm}(c) \cap \Hsr{1}.
\end{align*}
Here $\Lambda^{\pm}(c)$ denotes either $\Lambda^{+}(c)$ or $\Lambda^{-}(c)$.

\begin{lemma}\label{lemma:12}
For any $c \in (0,c_1)$ it holds that
	\begin{align*}
	\inf_{u \in \Lambda^{\pm}_r(c)} F(u) = \inf_{u \in \Lambda^{\pm}(c)} F(u).
	\end{align*}
	Also, if $\inf_{u \in \Lambda^{\pm}(c)} F(u)$ is reached, it is reached by a Schwarz symmetric function.
\end{lemma}

\begin{proof}
	Since $\Lambda^{\pm}_r(c) \subset \Lambda^{\pm}(c)$, we directly have
	\begin{align}
	\inf_{u \in \Lambda^{\pm}_r(c)} F(u) \geq \inf_{u \in \Lambda^{\pm}(c)} F(u). \label{eqn:24}
	\end{align}
	Therefore, it suffices to prove that 
	\begin{align}
	\inf_{u \in \Lambda^{\pm}_r(c)} F(u) \leq \inf_{u \in \Lambda^{\pm}(c)} F(u). \label{eqn:25t}
	\end{align}
	In this aim we start to note that
	\begin{align}
	\inf_{u \in \Lambda^{+}(c)} F(u) = \inf_{u \in S(c)} \min_{0 < t \leq s_{u}^+}F(u^t) \quad \text{and} \quad \inf_{u \in \Lambda^{-}(c)} F(u) = \inf_{u \in S(c)} \max_{s_u^+ < t  \leq s_{u}^-}F(u^t). \label{eqn:26l}
	\end{align}
	Now  let $u \in S(c)$ and $v \in S_r(c)$ be the  Schwarz rearrangement of $|u|$. Taking into account that $A(v) \leq A(u)$, $C(v) = C(u)$, and    
	by the Riesz's rearrangement inequality (see \cite[Section 3.7]{LiebLoss2001analysis}),  $B(v) \geq B(u)$, 
	we have for all $t>0,$ 
	\begin{align}
	F(v^t) = \dfrac{1}{2}t^2 A(v) - \dfrac{\gamma}{4} t B(v) - \dfrac{a}{p} t^{\sigma} C(v)
	\leq \dfrac{1}{2}t^2 A(u) - \dfrac{\gamma}{4} t B(u) - \dfrac{a}{p} t^{\sigma} C(u) \label{eqn:27l}
	= F(u^t). 
	\end{align}
	Observe that, for any $w \in S(c)$,
	\begin{equation*}
	g_w'(t) = t A(w) - \frac{\gamma}{4}B(w) - \frac{a \sigma}{p}t^{\sigma -1}C(w) \quad\mbox{and}\quad g_w''(t) =  A(w) - \frac{a \sigma(\sigma -1)}{p}t^{\sigma -2}C(w).
	\end{equation*}
	Thus we have
	\begin{align*}
	g_v'(0) \leq g_u'(0) <0 \quad \mbox{and} \quad g_v''(t) \leq g_u''(t), \quad \forall t >0. 
	\end{align*}
	This implies that $0 < \spu \leq s_v^{+} < s_v^{-} \leq \smu$. Hence, we deduce from \eqref{eqn:27l} that
	\begin{align*}
	\min_{0 < t \leq s_v^+}F(v^t) \leq \min_{0 < t \leq s_u^+}F(u^t) \qquad \text{and} \qquad \max_{s_v^+ < t \leq s_v^-}F(v^t) \leq \max_{s_u^+ < t \leq s_u^-}F(u^t).
	\end{align*}	
	In view of \eqref{eqn:26l}, the inequality \eqref{eqn:25t} holds. Now if $u_0 \in \Lambda^{+}(c)$ is such that $F(u_0) = \inf_{u \in \Lambda^{+}(c)} F(u)$ we see that $v$, the Schwarz rearrangement of $|u_0|$, belongs to $\Lambda^{+}_r(c)$. Indeed, if either $A(v) < A(u_0)$ or $B(v) > B(u_0)$ then $F(v^t) < F(u_0^t)$. Hence, in view of the above arguments, we get 
	$$\inf_{u \in \Lambda^{+}(c)} F(u) = \inf_{u \in S(c)} \min_{0 < t \leq s_{u}^+}F(u^t) \leq \min_{0 < t \leq s_{v}^+}F(v^t) < \min_{0 < t \leq s_{u_0}^+}F(u_0^t) = \inf_{u \in \Lambda^{+}(c)} F(u)$$
	a contradiction. Thus $A(v) = A(u_0)$, $B(v) = B(u_0)$ and $C(v) = C(u_0)$ from which we deduce that $v \in \Lambda^{+}_r(c)$ and $F(v) = F(u_0).$ 
	The case of $u_0 \in \Lambda^{-}(c)$ such that $F(u_0) = \inf_{u \in \Lambda^{-}(c)} F(u)$ is treated similarly. This ends the proof of the lemma.
\end{proof}

Recalling that $\gamma^{+ }(c)$ and $\gamma^{+ }(c)$ are defined in \eqref{eqn:gamma} we have
\begin{lemma} \label{lemma:9}
	For any $c \in (0, c_1)$, there exists a bounded Palais-Smale sequence $(u_n^{+}) \subset \Lambda^{+}_r(c)$ for $F$ restricted to $S(c)$ at level $\gamma^{+}(c)$ and a bounded Palais-Smale sequence $(u_n^{-}) \subset \Lambda^{-}_r(c)$ for $F$ restricted to $S(c)$ at level $\gamma^{-}(c)$.
\end{lemma}

In order to prove \cref{lemma:9} we define the functions
\begin{align*}
	&I^{+}: S(c) \to \R, \quad I^{+}(u) = F(u^{\spu}), \\
	&I^{-}: S(c) \to \R, \quad I^{-}(u) = F(u^{\smu}).
\end{align*}
Note that since the maps $u \mapsto \spu$ and $u \mapsto \smu$ are of class $\mathbb{C}^1$, see \cref{lemma:6}, the functionals $I^{+}$ and $I^{-}$ are of class $\mathbb{C}^1$. 
\begin{lemma} \label{lemma:7}
	For any $c \in (0, c_1)$, we have that $dI^{+}(u)[\psi] = dF(u^{\spu})[\psi^{\spu}]$ and $dI^{-}(u)[\psi] = dF(u^{\smu})[\psi^{\smu}]$ for any $u \in S(c)$, $\psi \in T_uS(c)$.
\end{lemma}
\begin{proof}
	We first give the proof for $I^{+}$. Let $\psi \in T_uS(c)$, then $\psi = h'(0)$ where $h: (-\epsilon, \epsilon) \mapsto S(c)$ is a $\mathit{C}^1$-cure with $h(0) = u$. We consider the incremental quotient
	\begin{align}
		\dfrac{I^{+}(h(t)) - I^{+}(h(0))}{t} = \dfrac{F(h(t)^{s_t}) - F(h(0)^{s_0})}{t}, \label{eqn:3}
	\end{align}
	where $s_t:= s_{h(t)}^{+}$, and hence $s_0 = \spu$. Recalling from \cref{lemma:6} that $s_0$ is a strict local minimum of $s \mapsto F(u^s)$ and $u \mapsto s_0$ is continuous, we get 
	\begin{align*}
		&F(h(t)^{s_t}) - F(h(0)^{s_0})
		\geq F(h(t)^{s_t}) - F(h(0)^{s_t}) \\
		= &\dfrac{s_t^2}{2} \Big[A(h(t)) - A(h(0)) \Big] - \dfrac{\gamma s_t}{4} \Big[B(h(t)) - B(h(0)) \Big] -\dfrac{a s_t^{\sigma}}{p} \Big[C(h(t)) - C(h(0)) \Big]\\
		= &s_t^2 \intrb \diff h(\tau_1 t) \cdot \diff h'(\tau_1 t) t dx - \gamma s_t \intrb\intrb \dfrac{\abs{h(\tau_2 t)(x)}^2 h(\tau_2 t)(y) h'(\tau_2 t)(y)}{\abs{x-y}} dxdy 
		- a s_t^{\sigma} \intrb \abs{h(\tau_3 t)}^{p-2} h(\tau_3 t) h'(\tau_3 t) dx,
	\end{align*}
	for some $\tau_1, \tau_2, \tau_3 \in (0,1)$. Analogously
	\begin{align*}
	&F(h(t)^{s_t}) - F(h(0)^{s_0})
	\leq F(h(t)^{s_0}) - F(h(0)^{s_0}) \\
	= &s_0^2 \intrb \diff h(\tau_4 t) \cdot \diff h'(\tau_4 t) t dx - \gamma s_0 \intrb\intrb \dfrac{\abs{h(\tau_5 t)(x)}^2 h(\tau_5 t)(y) h'(\tau_5 t)(y)}{\abs{x-y}} dxdy 
	- a s_0^{\sigma} \intrb \abs{h(\tau_6 t)}^{p-2} h(\tau_6 t) h'(\tau_6 t) dx,
	\end{align*}
	for some $\tau_4, \tau_5, \tau_6 \in (0,1)$. Now, from \eqref{eqn:3} we deduce that
	\begin{align*}
		\lim_{t \to 0} \dfrac{I^{+}(h(t)) - I^{+}(h(0))}{t}
		 = &(\spu)^2 \intrb \diff u \diff \psi dx - \gamma (\spu) \intrb\intrb \dfrac{\abs{u(x)}^2 u(y) \psi(y)}{\abs{x-y}} dxdy - a (\spu)^{\sigma} \intrb \abs{u}^{p-2} u \psi dx \\
		 = & \intrb \diff u^{\spu} \diff \psi^{\spu} dx - \gamma \intrb\intrb \dfrac{\abs{u^{\spu}(x)}^2 u^{\spu}(y) \psi^{\spu}(y)}{\abs{x-y}} dxdy - a \intrb \abs{u^{\spu}}^{p-2} u^{\spu} \psi^{\spu} dx \\
		 = &dF(u^{\spu})[\phi^{\spu}],
	\end{align*}
	 for any $u \in S(c)$, $\psi \in T_uS(c)$. The proof for $I^{-}$ is similar.
\end{proof}


Let $\mathcal{G}$ be the set of all singletons belonging to $S_r(c)$. It is clearly a homotopy stable family of compact subsets of $S_r(c)$ with closed boundary (an empty boundary actually) in the sense of \cite[Definition 3.1]{ghoussoub1993duality}. In view of \cref{lemma:12} we have that

\begin{align*}
		e_{\mathcal{G}}^{+} : = \inf_{A \in \mathcal{G}} \max_{u \in A} I^{+}(u) = \inf_{u \in S_r(c)} I^{+}(u) =\inf_{u \in \Lambda^{+}_r(c)} F(u) = \inf_{u \in \Lambda^{+}(c)} F(u) = \gamma^{+}(c).
	\end{align*}
	\begin{align*}
		e_{\mathcal{G}}^{-} : = \inf_{A \in \mathcal{G}} \max_{u \in A} I^{-}(u) = \inf_{u \in S_r(c)} I^{-}(u) =\inf_{u \in \Lambda^{-}_r(c)} F(u) = \inf_{u \in \Lambda^{-}(c)} F(u) = \gamma^{-}(c).
	\end{align*}


\begin{lemma} \label{lemma:8}
	For any $c \in (0, c_1)$, there exists a Palais-Smale sequence $(u_n^{+}) \subset \Lambda^{+}(c)$ for $F$ restricted to $S_r(c)$ at level $e_{\mathcal{G}}^{+}$ and a Palais-Smale  sequence 
	$(u_n^{-}) \subset \Lambda^{-}(c)$ for $F$ restricted to $S(c)$ at level $e_{\mathcal{G}}^{-}$.
\end{lemma}
\begin{proof}
We first treat the case of $e_{\mathcal{G}}^{+}$. 
	Let $(D_n) \subset \mathcal{G}$ be such that
	\begin{align*}
		\max_{u \in D_n} I^{+}(u) < e_{\mathcal{G}}^{+} + \dfrac{1}{n},
	\end{align*}
	and consider the homotopy
	\begin{align*}
		\eta: [0,1] \times S(c) \mapsto S(c), \quad \eta(t, u) = u^{1 - t + t \spu}.
	\end{align*}
	From the definition of $\mathcal{G}$, we have
	 \begin{align*}
	 		E_n := \eta(\{1\} \times D_n) = \{ u^{\spu} : u \in D_n\} \in \mathcal{G}.
	 \end{align*}
	 \cref{lemma:6} implies that $E_n \subset \slp$ for all $n \in \N$. Let $v \in E_n$, i.e. $v = u^{\spu}$ for some $u \in D_n$, and hence $I^{+}(v) = I^{+}(u)$. So, we have
	 \begin{align*}
	 	\max_{v \in E_n} I^{+}(v) = \max_{u \in D_n} I^{+}(u).
	 \end{align*} 
	 Therefore, $E_n$ is another minimizing sequence for $e_{\mathcal{G}}^{+}$. Applying \cite[Theorem 3.2]{ghoussoub1993duality}, in the particular case where the boundary $B = \emptyset$, there exists a Palais-Smale sequence $(y_n)$ for $I^{+}$ on $S(c)$ at level $e_{\mathcal{G}}^{+}$ such that
	 \begin{align}
	 	dist_{\Ho} (y_n, E_n) \to 0 \quad \text{as } n \to \infty. \label{eqn:4}
	 \end{align} 
	 Now writing $s_n := s_{y_n}^{+}$ we set $u_n^{+} := y_n^{s_n} \in \slp$. We claim that there exists a constant $C > 0$ such that
	 \begin{align}
	 	\dfrac{1}{C} \leq s_n^2 \leq C \label{eqn:5}
	 \end{align}
	 for $n \in \N$ large enough. Indeed, notice first that
	 \begin{align*}
	 	s_n^2 = \dfrac{A(u_n^{+})}{A(y_n)}.
	 \end{align*}
	 By $F(u_n^{+}) = I^{+}(y_n) \to e_{\mathcal{G}}^{+} = \gamma^+(c) <0$ we deduce from \eqref{eqn:29} that there exists $M>0$ such that
	 \begin{align}\label{A-bounds}
	 	\dfrac{1}{M} \leq A(u_n^{+}) \leq M.
	 \end{align}
	 On the other hand, since $E_n \in \slp$ is a minimizing sequence for $e_{\mathcal{G}}^{+}$ and $F$ is $\Ho$ coercive on $\slp$, we obtain that $E_n$ is uniformly bounded in $\Ho$ and thus from \eqref{eqn:4}, it implies that $\sup_{n}A(y_n) < \infty$. Also, since $E_n$ is compact for every $n \in \N$, there exist a $v_n \in E_n$ such that $\normHs{v_n - y_n}{1} \to 0$ as $n \to 0$ due to \eqref{eqn:4}. Using \cref{lemma:5} again, we have, for a $\delta > 0$,
	 \begin{align*}
	 	A(y_n) \geq A(v_n) - A(v_n-y_n) \geq \dfrac{\delta}{2}.
	 \end{align*}
	 This proves the claim \eqref{eqn:5}. From \eqref{eqn:21}, and by \cref{lemma:2.4}, \cref{lemma:7}, we have 
	 \begin{align*}
	 \norm{{dF}_{|S(c)} (u_n^{+})}_{*} = \sup_{\norm{\psi} \leq 1, \psi \in T_uS(c)} \abs*{dF(u_n^{+})[\psi]} 
	 = \sup_{\norm{\psi} \leq 1, \psi \in T_uS(c)} \abs*{dF(u_n^{+})\[\(\psi^{\frac{1}{s_n}}\)^{s_n}\]}
	 = \sup_{\norm{\psi} \leq 1, \psi \in T_uS(c)} \abs*{dI^{+}(y_n)\[\psi^{\frac{1}{s_n}}\]}.
	 \end{align*} 
	This implies that $(u_n^{+}) \subset \slp$ is a Palais-Smale sequence for $F$ restricted to $S(c)$ at level $e_{\mathcal{G}}^{+}$ since $(y_n)$ is a Palais-Smale sequence for $I^{+}$ at level $e_{\mathcal{G}}^{+}$ and
	and $\norm*{\psi^{\frac{1}{s_n}}} \leq C_1 \norm{\psi} \leq C_1$ due to \eqref{eqn:5}. For the case of $e_{\mathcal{G}}^{-}$ the proof is identical except that we use \cref{lemma:16}\eqref{point:2} along with \eqref{eqn:29} to conclude that there exists a $M >0$ such that \eqref{A-bounds} holds for $A(u_n^{-})$ replacing $A(u_n^{+})$.
\end{proof}

\begin{proof}[Proof of \cref{lemma:9}]
Applying  \cref{lemma:8}, we deduce that there exists a Palais-Smale sequence $(u_n^{+}) \subset \Lambda_r^{+}(c)$ for $F$ restricted to $S(c)$ at level $e_{\mathcal{G}}^{+} = \gamma^{+}(c)$ and a Palais-Smale sequence $(u_n^{-}) \subset \Lambda_r^{-}(c)$ for $F$ restricted to $S(c)$ at level $e_{\mathcal{G}}^{-} = \gamma^{-}(c)$.  In both cases the boundedness of these sequences follows from \cref{lemma:5}. 
\end{proof}

\subsection{The compactness of our Palais-Smale sequences in the Sobolev subcritical case $ p \in (\frac{10}{3},6)$}\label{Subsection3-2}

\begin{lemma}\label{lemma:15}
Let $ p \in (\frac{10}{3},6)$.
	For any $c \in (0, c_1)$, if either $(u_n) \subset \Lambda^{+}(c)$ is a minimizing sequence for $\gamma^{+}(c)$  or $(u_n) \subset \Lambda^{-}(c)$ is a minimizing sequence for $\gamma^{-}(c)$, it weakly converges, up to translation, to a non-trivial limit.
\end{lemma}
\begin{proof}
	Since $F$ restricted to $\slc$ is coercive on $\Ho$ (see \cref{lemma:5}), $(u_n)$ is bounded. Hence, up to translation, $u_n \weakto u_c$ weakly in $\Ho$.
	Let us argue by contradiction assuming that $u_c = 0$, this means that $(u_n)$ is vanishing.
	By \cite[Lemma I.1]{LIONS1984part2}, we have, for $2 < q < 6$, 
	\begin{align*}
	\normLp{u_n}{q} \to 0, \quad \text{as } n \to \infty.
	\end{align*}
	This implies that
	\begin{align*}
	C(u_n) \to 0, \qquad \text{and} \qquad B(u_n) \leq K_1 \norm{u_n}_{\Lp{\frac{12}{5}}}^4 \to 0,
	\end{align*}
	due to \eqref{eqn:2.1}. 
	Since $(u_n) \subset \Lambda(c)$, we have $Q(u_n) = 0$, and hence
	\begin{align}
	A(u_n) = \dfrac{\gamma}{4} B(u_n) + \dfrac{a \sigma}{p} C(u_n) \to 0. \label{eqn:22}
	\end{align}
	If we assume that  $(u_n) \subset \Lambda^{-}(c)$ we recall that by \cref{lemma:16}, there exists $\alpha > 0$ such that 
	\begin{align*}
	A(u_n) \geq \alpha > 0, \quad \forall n \in \N,
	\end{align*}
	contradicting \eqref{eqn:22}. If we assume that $(u_n) \subset \Lambda^{+}(c)$ then since
	\begin{align*}
	F(u_n) = \dfrac{1}{2} A(u_n) - \dfrac{\gamma}{4} B(u_n) - \dfrac{a}{p} C(u_n) \to 0 
	\end{align*}
	we reach a contradiction with the fact that
	\begin{align*}
	F(u_n) \to \gamma^{+}(c) = \inf_{u \in \slp} F(u) < 0.
	\end{align*}
	The lemma is proved.
\end{proof}

\begin{lemma} \label{lemma:11}
Let $ p \in (\frac{10}{3},6)$.
	Assume that a bounded Palais-Smale sequence $(u_n) \subset \Lambda_r(c)$ for $F$ restricted to $S(c)$ is weakly convergent, up to translation, to the nonzero function $u_c$. Then, up to translation, $u_n \to u_c \in \Lambda_r(c)$ strongly in $\Hsr{1}$. In particular $u_c$ is a radial solution to \eqref{eqn:Laplace} for some $\lambda_c > 0$ and $\normLp{u_c}{2}^2 = c$.
\end{lemma}
\begin{proof}
	Since the embedding $\Hsr{1} \subset \Lp{q}$ is compact for $q \in (2, 6)$, see \cite{Strauss1977} and, up to translation, $u_n \weakto u_c$ weakly in $\Hsr{1}$, we have, up to translation, $u_n \to u_c$ strongly in $\Lp{q}$ for $q \in (2, 6)$ and a.e in $\Rb$. 
	
	Since $(u_n) \subset \Ho$ is bounded, following \cite[Lemma 3]{BerestyckiLions1983_2}, we know that
	\begin{align*}
		F'_{|S(c)}(u_n) \to 0 \text{ in } \Hs{-1} 
		\iff F'(u_n) - \dfrac{1}{c}\inner{F'(u_n)}{u_n}u_n \to 0 \text{ in } \Hs{-1}.
	\end{align*}
	Thus, for any $w \in \Ho$, we have
	\begin{align}
	\begin{split}
		o_n(1) &= \inner*{F'(u_n) - \dfrac{1}{c}\inner{F'(u_n)}{u_n}u_n}{w} \\
		&= \intrb \diff u_n \diff w dx + \lambda_n \intrb u_n w dx 
		- \gamma \intrb\intrb \dfrac{\abs{u_n(x)}^2 u_n(y) w(y)}{\abs{x-y}} dxdy - a \intrb \abs{u_n}^{p-2} u_n w dx ,
	\end{split} \label{eqn:19}
 	\end{align}
 	where $o_n(1) \to 0$ as $n \to \infty$ and 
 	\begin{align*}
 		\lambda_n = \dfrac{-1}{c} \[A(u_n) - \gamma B(u_n) - aC(u_n)\]
 		= \dfrac{1}{c} \[\dfrac{3\gamma}{4}B(u_n) + a \(1 - \dfrac{\sigma}{p}\) C(u_n) \],
 	\end{align*}
 	due to $Q(u_n) = 0$. Since $u_n \in \Hsr{1}$, we have $C(u_n) \to C(u_c)$ and $B(u_n) \to B(u_c)$ (see \cref{lemma:2.5}). Hence, we obtain that
 	\begin{align*}
 	\lambda_n \to \lambda_c =  \dfrac{1}{c} \[\dfrac{3\gamma}{4}B(u_c) + a \(1 - \dfrac{\sigma}{p}\) C(u_c) \].
 	\end{align*}
 	Now, using \cite[Lemma 2.2]{ZHAO2008}, the equation \eqref{eqn:19} leads to
 	\begin{align}
 		\intrb \diff u_c \diff w dx + \lambda_c \intrb u_c w dx - \gamma \intrb\intrb \dfrac{\abs{u_c(x)}^2 u_c(y) w(y)}{\abs{x-y}} dxdy - a \intrb \abs{u_c}^{p-2} u_c w dx = 0 \label{eqn:20}
 	\end{align}
 	due to the weak convergence in $\Hsr{1}$ and $\lambda_n \to \lambda_c \in \R$.
 	This implies that $(u_c, \lambda_c)$ satisfies
 	\begin{align*}
 	- \laplace u_c + \lambda_c u_c - \gamma (\abs{x}^{-1} * \abs{u_c}^2) u_c  - a \abs{u_c}^{p-2}u_c = 0 \quad\text{ in } \Hs{-1}.
 	\end{align*}
 	By the assumption $u_c \neq 0$ and by \cref{lemma:2.3}, we obtain that $Q(u_c) = 0$ and $\lambda_c > 0$.
 	
 	Now choosing $w = u_n$ in \eqref{eqn:19} and choosing $w = u_c$ in \eqref{eqn:20}, we obtain that 
 	\begin{align*}
 		\intrb \abs{\diff u_n}^2 dx + \lambda_n \intrb \abs{u_n}^2 dx - \gamma B(u_n) - a C(u_n) 
 		\to \intrb \abs{\diff u_c}^2 dx + \lambda_c \intrb \abs{u_c}^2 dx - \gamma B(u_c) - a C(u_c).
 	\end{align*}
 	We can deduce from $B(u_n) \to B(u_c)$, $C(u_n) \to C(u_c)$ and $\lambda_n \to \lambda_c$ that
 	\begin{align*}
 		\intrb \abs{\diff u_n}^2 dx + \lambda_c \intrb \abs{u_n}^2 dx \to \intrb \abs{\diff u_c}^2 dx + \lambda_c \intrb \abs{u_c}^2 dx.
 	\end{align*}
 	Since $\lambda_c > 0$, we conclude that $u_n \to u_c$ strongly in $\Hsr{1}$. The lemma is proved.	
\end{proof}

\begin{proof}[Proof of \cref{theorem:1} in the subcritical case $ p \in (\frac{10}{3},6)$.]
We give the proof for $\gamma^+(c)$, the treatment for $\gamma^-(c)$ is identically.	For any $c \in (0, c_1)$, by \cref{lemma:9}, there exists a bounded Palais-Smale sequence $(u_n^{+}) \subset \Lambda_r^{+}(c)$ for $F$ restricted to $S(c)$ at level $\gamma^{+}(c)$. From \cref{lemma:15} and \cref{lemma:11}, we deduce that $u_n^{+} \to u_c^{+} \in \Lambda_r(c)$ strongly in $\Hsr{1}$ and that there exists $\lambda_c^{+} > 0$ such that $(u_c^{+}, \lambda_c^{+})$ is a solution to \eqref{eqn:Laplace}. Since $\slz = \emptyset$ (see \cref{lemma:4}), we conclude that $u_c^{+} \in \Lambda_r^{+}(c)$. From \cref{lemma:12} we can thus assume that $u_c^{+}$ is a Schwarz symmetric function. Hence, $u_c^{+}$ is non-negative. At this point, we can deduce from \cref{lemma:bounded-positive-solution} that $u_c^{+}$ is  a bounded continuous positive function.
\end{proof}


\subsection{The compactness of our Palais-Smale sequences in the Sobolev critical case $p= 6$}\label{Subsection3-3} $ $

Our next lemma is directly inspired from \cite[Proposition 3.1]{Soave2019Sobolevcriticalcase}.
\begin{lemma} \label{lemma:18}
	Let $c \in (0,c_1)$ and $(u_n) \subset \Lambda_r^{+}(c)$ or $(u_n) \subset \Lambda_r^{-}(c)$  be a Palais-Smale sequence for $F$ restricted to $S(c)$ at level $m \in \R$ 
	which is weakly convergent, up to subsequence, to the function $u_c$. If $(u_n) \subset \Lambda_r^{+}(c)$ we assume that $m \neq 0$ and if $(u_n) \subset \Lambda_r^{-}(c)$ we assume that
	\begin{align}
		m < \dfrac{1}{3 \sqrt{a K_{GN}}}. \label{eqn:36}
	\end{align}
	Then $u_c \neq 0$  and we have the following alternative:
	\begin{enumerate}[label=(\roman*), ref=\roman*]
		\item\label{point:18i} either
		\begin{align}
			F(u_c) \leq m - \dfrac{1}{3 \sqrt{a K_{GN}}}, \label{eqn:32}
		\end{align}
		\item\label{point:18ii} or 
		\begin{align}
			u_n \to u_c \quad \mbox{strongly in} \quad H_r^1(\R^3). \label{eqn:33}
		\end{align}
	\end{enumerate}
\end{lemma}

\begin{proof}
	Since $u_n \weakto u_c$ weakly in $\Hsr{1}$, we have, up to subsequence, $u_n \to u_c$ strongly in $\Lp{q}$ for $q \in (2, 6)$ and a.e in $\Rb$. 
	
	Let us first show that $u_c \neq 0$. We argue by contradiction  assuming that $u_c = 0$, this means that $(u_n)$ is vanishing. By \cite[Lemma I.1]{LIONS1984part2}, we have, for $2 < q < 6$, 
	\begin{align*}
	\normLp{u_n}{q} \to 0, \quad \text{as } n \to \infty.
	\end{align*}
	This implies from \eqref{eqn:2.1} that
	\begin{align*}
	B(u_n) \leq K_1 \norm{u_n}_{\Lp{\frac{12}{5}}}^4 \to 0.
	\end{align*} 
	Since $(u_n) \subset \Lambda(c)$, we have
	\begin{align*}
	A(u_n) = a C(u_n) + o_n(1).
	\end{align*}
	Passing to the limit as $n \to \infty$, up to subsequence we infer that
	\begin{align*}
	\lim_{n \to \infty} A(u_n) = \lim_{n \to \infty} a C(u_n) :=\ell \geq 0.
	\end{align*}
	Using \cref{lemma:1}\eqref{point:1ii}, we have
	\begin{align*}
	\ell = \lim_{n \to \infty} a C(u_n) \leq \lim_{n \to \infty} a K_{GN} [A(u_n)]^3 = a K_{GN} \ell^3.
	\end{align*}
	Therefore, either $\ell = 0$ or $\ell \geq (a K_{GN})^{-\frac{1}{2}}$. If $(u_n) \subset \Lambda^+(c)$, we have $A(u_n) > 5aC(u_n)$, and then $\ell=0$. This implies that $F(u_n) \to 0$ and this contradicts the assumption that $m \neq 0$. Also, if $(u_n) \subset \Lambda^-(c)$,  \cref{lemma:16}\eqref{point:2} ensure that $\ell \geq (a K_{GN})^{-\frac{1}{2}}$. Hence, we have
	\begin{align*}
		m + o_n(1) = F(u_n) &= \dfrac{\sigma - 2}{2\sigma} A(u_n) - \dfrac{\gamma(\sigma - 1)}{4\sigma} B(u_n)
		= \dfrac{1}{3} A(u_n) + o_n(1) = \dfrac{1}{3}\ell + o_n(1) \geq \dfrac{1}{3 \sqrt{a K_{GN}}} + o_n(1),
 	\end{align*}
	which contradicts our assumption \eqref{eqn:36}. Thus, we have that $u_c \neq 0$.

	Now, since $(u_n) \subset \Ho$ is bounded, following \cite[Lemma 3]{BerestyckiLions1983_2}, we know that
	\begin{align*}
	F'_{|S(c)}(u_n) \to 0 \text{ in } \Hs{-1} 
	\iff F'(u_n) - \dfrac{1}{c}\inner{F'(u_n)}{u_n}u_n \to 0 \text{ in } \Hs{-1}.
	\end{align*}
	Thus, for any $w \in \Ho$, we have
	\begin{align}
	\begin{split}
	o_n(1) &= \inner*{F'(u_n) - \dfrac{1}{c}\inner{F'(u_n)}{u_n}u_n}{w}\\
	&= \intrb \diff u_n \diff w dx + \lambda_n \intrb u_n w dx 
	 - \gamma \intrb\intrb \dfrac{\abs{u_n(x)}^2 u_n(y) w(y)}{\abs{x-y}} dxdy - a \intrb \abs{u_n}^{p-2} u_n w dx ,
	\end{split} \label{eqn:26}
	\end{align}
	where $o_n(1) \to 0$ as $n \to \infty$ and 
	\begin{align*}
	\lambda_n = \dfrac{-1}{c} \[A(u_n) - \gamma B(u_n) - aC(u_n)\]
	= \dfrac{3\gamma}{4c} B(u_n),
	\end{align*}
	due to $Q(u_n) = 0$. By $B(u_n) \to B(u_c)$ (see \cref{lemma:2.5}), we obtain that
	\begin{align}\label{key-point}
		\lambda_n \to \lambda_c = \dfrac{3\gamma}{4c} B(u_c).
	\end{align}
	Now, using \cite[Lemma 2.2]{ZHAO2008}, the equation \eqref{eqn:26} leads to
	\begin{align}
	\intrb \diff u_c \diff w dx + \lambda_c \intrb u_c w dx - \gamma \intrb\intrb \dfrac{\abs{u_c(x)}^2 u_c(y) w(y)}{\abs{x-y}} dxdy - a \intrb \abs{u_c}^{p-2} u_c w dx = 0 \label{eqn:27}
	\end{align}
	due to the weak convergence in $\Hsr{1}$ and $\lambda_n \to \lambda_c \in \R$.
	This implies that $(u_c, \lambda_c)$ satisfies
	\begin{align*}
	- \laplace u_c + \lambda_c u_c - \gamma (\abs{x}^{-1} * \abs{u_c}^2) u_c  - a \abs{u_c}^{p-2}u_c = 0 \quad\text{ in } \Hs{-1}.
	\end{align*}
	By \cref{lemma:2.3}, we obtain that $Q(u_c) = 0$ and $\lambda_c >0$.
	
	Let $v_n:= u_n - u_c \weakto 0$ in $\Hsr{1}$. By Brezis-Lieb lemma (see \cite{BrezisLieb1983}), we obtain that
	\begin{align}
		A(u_n) = A(u_c) + A(v_n) + o_n(1), \qquad C(u_n) = C(u_c) + C(v_n) + o_n(1). \label{eqn:28}
	\end{align}
	By $B(u_n) \to B(u_c)$ (see \cref{lemma:2.5}) and by $Q(u_n) = 0$, we have
	\begin{align*}
		A(u_c) + A(v_n) - \dfrac{\gamma}{4} B(u_c) - a [C(u_c) - C(v_n)] = o_n(1).
	\end{align*}
	Taking into account that $Q(u_c) = 0$, we get $A(v_n) = a C(v_n) + o_n(1)$. Passing to the limit as $n \to \infty$, up to subsequence we infer that 
	\begin{align*}
		\lim_{n \to \infty} A(v_n) = \lim_{n \to \infty} a C(v_n) :=k \geq 0.
	\end{align*}
	Using \cref{lemma:1}\eqref{point:1ii}, we have
	\begin{align*}
		k = \lim_{n \to \infty} a C(v_n) \leq \lim_{n \to \infty} a K_{GN} [A(v_n)]^3 = a K_{GN} k^3.
	\end{align*}
	Therefore, either $k = 0$ or $k \geq (a K_{GN})^{-\frac{1}{2}}$. 

	If $k \geq (a K_{GN})^{-\frac{1}{2}}$, then by \eqref{eqn:28} and by $B(u_n) \to B(u_c)$, we have
	\begin{align*}
		m = \lim_{n \to \infty} F(u_n) &= \lim_{n \to \infty} \[\dfrac{1}{2} A(u_n) - \dfrac{\gamma}{4} B(u_n) - \dfrac{a}{6} C(u_n) \]\\
		 &= \lim_{n \to \infty} \[\dfrac{1}{2} A(u_c) + \dfrac{1}{2} A(v_n) - \dfrac{\gamma}{4} B(u_c) - \dfrac{a}{6}C(u_c) -\dfrac{a}{6}C(v_n) \] \\
		 &= F(u_c) + \dfrac{1}{3} k 
		 \geq F(u_c) + \dfrac{1}{3} (a K_{GN})^{-\frac{1}{2}}.
	\end{align*}
	This implies that alternative \eqref{point:18i} holds. 
	
	If instead $k = 0$, then by \eqref{eqn:28}, we have $A(u_n) \to A(u_c)$ and $C(u_n) \to C(u_c)$. Choosing $w = u_n$ in \eqref{eqn:26} and $w = u_c$ in \eqref{eqn:27}, we obtain that
	\begin{align*}
		A(u_n) + \lambda_n \normLp{u_n}{2}^2 - \gamma B(u_n) - a C(u_n)
		\to A(u_c) + \lambda_c \normLp{u_c}{2}^2 - \gamma B(u_c) - a C(u_c).
	\end{align*}
	This implies that $\normLp{u_n}{2} \to \normLp{u_c}{2}$. Thus, we conclude that $u_n \to u_c$ strongly in $\Hsr{1}$.	
\end{proof}

\begin{proof}[Proof of \cref{theorem:1} in the critical case $p=6$ for $\gamma^+(c)$.]
	Since $\gamma^+(c) <0$, the fact that it is reached is a direct consequence of \cref{lemma:9},  \cref{lemma:18} and of the property, which is established in \cref{lemma:28}(iii) to come, that the map $c \mapsto \gamma^+(c)$ is non-increasing. The rest of the proof is identical to the one in the case $p \in (\frac{10}{3},6)$.
\end{proof}
In the rest of this subsection, we shall prove \cref{theorem:1} in the critical case $p=6$ for $\gamma^-(c)$.

\begin{lemma}\label{lemma:19}
Let $c \in (0,c_1)$. If 
\begin{align}\label{key-bound}
	 \gamma^-(c) < \gamma^+(c)  + \dfrac{1}{3 \sqrt{a K_{GN}}}
\end{align}
then there exists a $u_c \in \Lambda^{-}_{r}(c)$ with $F(u_c) = \gamma^-(c)$ which  is a radial solution to \eqref{eqn:Laplace} for some $\lambda_c > 0$ with $\normLp{u_c}{2}^2 = c$.
\end{lemma}

\begin{proof}
By \cref{lemma:9} there exists a Palais-Smale sequence $(u_n) \subset \Lambda^-(c)$ for $F$ restricted to $S(c)$ at the level $\gamma^-(c)$. If \eqref{key-bound} holds then necessarily \eqref{eqn:36}, with $m = \gamma^-(c)$ holds, and \eqref{eqn:32}  cannot holds. We deduce from \cref{lemma:18} that $u_n \to u_c$ strongly in $H^1(\R^N)$ and the conclusions follow.
\end{proof}

Now we shall show that 
\begin{lemma} \label{lemma:20}
	For any $c \in (0,c_1)$, we have that
	\begin{align*}
		\gamma^-(c) < \gamma^+(c)  + \dfrac{1}{3 \sqrt{a K_{GN}}}.
	\end{align*}
\end{lemma}
As already indicated our proof is inspired by \cite[Lemma 3.1]{Wei-Wu2021}. 
Let $u_{\varepsilon}$ be an extremal function for the Sobolev inequality in $\Rn$ defined by
\begin{align}\label{Def-extremal}
u_{\varepsilon}(x) := \dfrac{[N(N-2)\varepsilon^2]^{\frac{N-2}{4}}}{[\varepsilon^2+|x|^2]^{\frac{N-2}{2}}}, \quad \varepsilon > 0, \quad x\in\Rn.
\end{align}
Let $\xi \in C_0^{\infty}(\R^N)$ be a radial non-increasing cut-off function with $\xi \equiv 1$ in $B_1$, $\xi \equiv 0$ in $\R^N \backslash B_2$.
Setting $U_{\varepsilon}(x) = \xi(x) u_{\varepsilon}(x)$ we recall the following result, see \cite[Lemma 7.1]{JeanjeanLe2020}. 

\begin{lemma}\label{estimates-U}
	Denoting $\omega$ the area of the unit sphere in $\Rb$, we have
	\begin{enumerate}[label=(\roman*)]	
		\item \label{point:estimates-U:1} \begin{align*} 
			\normLp{\diff U_{\varepsilon}}{2}^2 = \sqrt{\dfrac{1}{K_{GN}}} + O(\varepsilon) \quad \mbox{and} \quad \normLp{U_{\varepsilon} }{2^*}^{2^*} = \sqrt{\dfrac{1}{K_{GN}}} + O(\varepsilon^3).
		\end{align*}
		\item \label{point:estimates-U:3} For some positive constant $K >0$,
		\begin{align*}  
			\normLp{U_{\varepsilon} }{q}^{q} = 
			\begin{cases}
				K \varepsilon^{3-\frac{q}{2}} + o(\varepsilon^{3-\frac{q}{2}}) \quad &\mbox{if } q \in (3,6), \\
				\omega \var^{\frac{3}{2}} |\log \var| + O(\var^{\frac{3}{2}}) &\mbox{if } q = 3, \\
				\omega \Big(\int_0^2 \frac{\xi^q(r)}{r^{q-2}} dr \Big) \var^{\frac{q}{2}} + o(\var^{\frac{q}{2}}) &\mbox{if } q \in [1,3).
			\end{cases}
		\end{align*}
	\end{enumerate}
\end{lemma}

In the rest of the subsection we assume that $c \in (0,c_1)$ is arbitrary but fixed. Let $u_c^+$ be as provided by \cref{theorem:1}. We recall that 
$u_c^+ \in \slp$  satisfies $F(u_c^+) = \gamma^+(c)$ and is a bounded continuous positive Schwarz symmetric function.

\begin{lemma} \label{lemma:21}
	For any $1 \leq p, q < \infty$, it holds that
	\begin{align*}
		\intrn |u_c^+(x)|^p |U_{\varepsilon}(x)|^q dx \sim \intrn |U_{\varepsilon}(x)|^q dx.
	\end{align*}
\end{lemma}
\begin{proof}
On one hand, since $u_c^+$ is bounded, we have that
	\begin{align*}
		\intrn |u_c^+(x)|^p |U_{\varepsilon}(x)|^q dx \leq \normLp{u_c^+ }{\infty}^{p} \intrn |U_{\varepsilon}(x)|^q dx.
	\end{align*}
	On the other hand, since $u_c^+ >0$ on $\R^3$ is continuous and the function $U_{\varepsilon}$ is compactly supported in $B_2$, we have that
	\begin{align*}
		\intrn |u_c^+(x)|^p |U_{\varepsilon}(x)|^q dx & =  \int_{B_2} |u_c^+(x)|^p |U_{\varepsilon}(x)|^q dx
		\geq \min_{x \in B_2} |u_c^+(x)|^p \int_{B_2}|U_{\varepsilon}(x)|^q dx  = \min_{x \in B_2} |u_c^+(x)|^p \intrn|U_{\varepsilon}(x)|^q dx.
	\end{align*}
	The lemma is proved.
\end{proof}

For any $\varepsilon >0$ and any $t>0$, we have
\begin{align}\label{LA}
	A(u_c^+ + t U_{\varepsilon}) = \normLp{\diff (u_c^+ + t U_{\varepsilon})}{2}^2 = A(u_c^+) + 2  \intrn \diff u_c^+(x) \cdot \diff (t U_{\varepsilon}(x)) \, dx  +    A(tU_{\varepsilon})
\end{align}
and
\begin{equation}\label{norm-L2}
	\normLp{ u_c^+ + t U_{\varepsilon}}{2}^2 = c +  2 \intrn  u_c^+(x)  (tU_{\varepsilon}(x)) dx  +    \normLp{t U_{\varepsilon}}{2}^2.
\end{equation}
Using that, for all $a,b \geq 0$, $(a+b)^6 \geq a^6 + 6a^5b + 6ab^5 + b^6$, and that both $u_c^+ \in H^1(\R^N)$ and $U_{\varepsilon}$ are non negative,  we readily derive that
\begin{align}\label{LC}
	C(u_c^+ + t U_{\varepsilon}) = \normLp{u_c^+ + t U_{\varepsilon} }{6}^{6} \geq C(u_c^+) + C(t U_{\varepsilon}) + 6 \intrn  (u_c^+(x))^5  (tU_{\varepsilon}(x)) dx + 6 \intrn  u_c^+(x)  (tU_{\varepsilon}(x))^5 dx .
\end{align}
Also, still using that $u_c^+ \in H^1(\R^N)$ and $U_{\varepsilon}$ are non negative, we get by direct calculations that
\begin{align}\label{LB}
	\begin{split}
		B(u_c^+ + t U_{\varepsilon}) 
		& = \intrb\intrb \dfrac{\abs{u_c^+(x) + t U_{\varepsilon}(x)}^2 \abs{u_c^+(y) + t U_{\varepsilon}(y)}^2}{\abs{x-y}} dxdy\\
		& \geq B(u_c^+) +  B(t U_{\varepsilon}) + 4 \intrb\intrb \dfrac{\abs{u_c^+(x)}^2 u_c^+(y) (t U_{\varepsilon}(y))}{\abs{x-y}} dxdy.
	\end{split}
\end{align}
Finally, since $u_c^+$ is solution of the following equation 
\begin{align*}
	- \laplace u + \lambda_c^+ u - \gamma (\abs{x}^{-1} * \abs{u}^2) u  - a \abs{u}^{p-2}u = 0 \quad \text{in } \Rb
\end{align*}
for a $\lambda_c^+ >0$, we have that
\begin{align}\label{LE}
	\begin{split}
		-\lambda_c^+ \intrn  u_c^+(x)  (tU_{\varepsilon})(x) dx 
		& = \intrn \diff u_c^+(x) \diff (tU_{\varepsilon})(x) dx \\
		& - \gamma \intrb\intrb \dfrac{\abs{u_c^+(x)}^2 u_c^+(y) (t U_{\varepsilon}(y))}{\abs{x-y}} dxdy - a  \intrn (u_c^+(x))^5  (tU_{\varepsilon}(x)) dx.
	\end{split}
\end{align}
Now, we define for $t>0$, $w_{\varepsilon,t}= u_c^+ + t U_{\varepsilon}$ and $\overline{w}_{\varepsilon, t}(x)= \sqrt{\theta}w_{\varepsilon, t}(\theta x)$ with $\theta^2 = \dfrac{1}{c} \normLp{ w_{\varepsilon, t}}{2}^2$. The proof of \cref{lemma:20} will follow directly from the three lemmas below. 

\begin{lemma} \label{lemma:31}
	It holds that
	\begin{equation*}
		\gamma^-(c) \leq \sup_{t \geq 0} F(\overline{w}_{\varepsilon, t})
	\end{equation*}	
	for $\varepsilon >0$ sufficiently small.
\end{lemma}
\begin{lemma} \label{lemma:32}
	There exist a $\varepsilon_0 >0$ and $0 < t_0 < t_1 <\infty$ such that
	\begin{equation*}
		F(\wbar) < \gamma^+(c) + \dfrac{1}{6 \sqrt{a K_{GN}}} 
	\end{equation*}
	for $t \not\in [t_0, t_1]$ and any $\varepsilon \in (0, \varepsilon_0]$.
\end{lemma}
\begin{lemma} \label{lemma:33}
 It holds that
	\begin{equation*}
		\max_{t \in [t_0, t_1]} F(\wbar) < \gamma^+(c) + \dfrac{1}{3 \sqrt{a K_{GN}}},
	\end{equation*}
	for any $\varepsilon \in (0, \varepsilon_0]$ where $\varepsilon_0$ and $t_0, t_1$ are provided by \cref{lemma:32}.
\end{lemma}

\begin{proof}[Proof of \cref{lemma:31}]
	By direct calculation we get
	\begin{equation}\label{L1}
		A(\overline{w}_{\varepsilon, t}) = A (w_{\varepsilon,t}), \quad C(\overline{w}_{\varepsilon, t})= C(w_{\varepsilon,t}),
	\end{equation}
	and 
	\begin{equation}\label{L11}
		\normLp{ \overline{w}_{\varepsilon, t}}{2}^2 = \theta^{-2} \normLp{ w_{\varepsilon, t}}{2}^2, \quad  B(\overline{w}_{\varepsilon, t})= \theta^{-3}B(w_{\varepsilon,t}).
	\end{equation}
	Since $\theta^2 = \dfrac{1}{c} \normLp{ w_{\varepsilon, t}}{2}^2$, we have that $\overline{w}_{\varepsilon, t} \in S(c)$. By \cref{lemma:6} there exists $s_{\varepsilon, t}^- >0$ such that $(\overline{w}_{\varepsilon, t})^{s_{\varepsilon, t}^-}  \in \slm.$ 
	We claim that $s_{\varepsilon, t}^- \to 0$ as $t \to + \infty$ uniformly for $\varepsilon >0$ sufficiently small. 
	Indeed, we have
	$$ A((\overline{w}_{\varepsilon, t})^{s_{\varepsilon, t}^-}) = \frac{\gamma}{4}B((\overline{w}_{\varepsilon, t})^{s_{\varepsilon, t}^-}) + a C((\overline{w}_{\varepsilon, t})^{s_{\varepsilon, t}^-})$$
	or equivalently
	\begin{equation*}
		(s_{\varepsilon, t}^-) A(\overline{w}_{\varepsilon, t}) 
		= \frac{\gamma}{4} B(\overline{w}_{\varepsilon, t}) + a (s_{\varepsilon, t}^-)^5 C(\overline{w}_{\varepsilon, t}).
	\end{equation*}	
	This implies that
	\begin{align} \label{T1}
	 A(\overline{w}_{\varepsilon, t}) \geq a (s_{\varepsilon, t}^-)^4 C(\overline{w}_{\varepsilon, t}).
	\end{align}	
	In view of \eqref{LA}, \eqref{L1}, \cref{estimates-U}\ref{point:estimates-U:1} and using H\"{o}lder's inequality, we have
	\begin{align} \label{T2}
		\begin{split}		
		A(\overline{w}_{\varepsilon, t}) &= A({w}_{\varepsilon, t}) = A(u_c^+) + 2  \intrn \diff u_c^+(x) \cdot \diff (t U_{\varepsilon}(x)) dx + A(t U_{\varepsilon})\\
		&\leq A(u_c^+) + 2 t  \normLp{\diff u_c^+}{2} \normLp{\diff U_{\varepsilon}}{2} + t^2 A(U_{\varepsilon}) 
		\to A(u_c^+) + 2 J \sqrt{A(u_c^+)} \, t   + J t^2 \mbox{ as } \varepsilon \to 0.
	\end{split}
	\end{align}
	In view of \eqref{LC}, \eqref{L1}  and  \cref{estimates-U}\ref{point:estimates-U:1}, we also have
	\begin{align}\label{T3}
		C(\overline{w}_{\varepsilon, t}) = C({w}_{\varepsilon, t}) \geq C(t U_{\varepsilon} ) = t^6 C(U_{\varepsilon} ) \to L t^6 \mbox{ as } \varepsilon \to 0.
	\end{align}
	Combining \eqref{T1}-\eqref{T3}, we obtain that, for $\varepsilon >0$ sufficiently small 
	\begin{align*}
		 A(u_c^+) + J \sqrt{A(u_c^+)} \, t   + J t^2 \geq a (s_{\varepsilon, t}^-)^4 L t^6,
	\end{align*} 
	which implies the claim. Since $\overline{w}_{\varepsilon, 0} = {w}_{\varepsilon, 0} = u_c^+$ and $u_c^+ \in \slp$  we obtain, see \cref{lemma:6}, that $s_{\varepsilon, 0}^- >1$. Still by \cref{lemma:6}, the map $t \mapsto s_{\varepsilon, t}^-$ is continuous which implies that there exists $t_{\varepsilon} >0$ such that $s_{\varepsilon, t_{\varepsilon}}^- =1$. It follows that $\overline{w}_{\varepsilon, t_{\varepsilon}} \in \slm$ and thus
	\begin{equation*}
		\sup_{t \geq 0}F(\overline{w}_{\varepsilon, t}) \geq F(\overline{w}_{\varepsilon, t_{\varepsilon}}) \geq \gamma^-(c). 
	\end{equation*}
	The lemma is proved.
\end{proof}

\begin{proof}[Proof of \cref{lemma:32}]
	In view of \eqref{L1} and \eqref{L11}, we have that
	\begin{align*}
		F(\overline{w}_{\varepsilon, t})  
		 =  \frac{1}{2}  A (w_{\varepsilon,t}) - \frac{\gamma}{4} \theta^{-3}B( w_{\varepsilon,t}) - \frac{a}{6} C (w_{\varepsilon,t}).
	\end{align*}
	Hence, by  \eqref{LA}, \eqref{LC} and \eqref{LB}, we get that
	\begin{align*}
		F(\overline{w}_{\varepsilon, t}) &\leq \frac{1}{2} \Big[ A(u_c^+) + 2  \intrn \diff u_c^+(x) \cdot \diff (t U_{\varepsilon}(x)) \, dx  +    A(tU_{\varepsilon}) \Big] - \frac{\gamma}{4} \theta^{-3} B(u_c^+) - \frac{a}{6} \Big[C(u_c^+)  + C(t U_{\varepsilon})\Big] \\
		&= F(u_c^+) + \dfrac{\gamma}{4}(1 - \theta^{-3}) B(u_c^+) + \intrn \diff u_c^+(x) \cdot \diff (t U_{\varepsilon}(x)) \, dx +  \frac{1}{2} A(tU_{\varepsilon}) - \frac{a}{6} C(t U_{\varepsilon}) \\
		& \leq F(u_c^+) + \dfrac{\gamma}{4}(1 - \theta^{-3}) B(u_c^+) + t  \normLp{\diff u_c^+}{2} \normLp{ \diff U_{\varepsilon}}{2} +  \frac{1}{2} t^2 A(U_{\varepsilon}) - \frac{a}{6} t^6 C( U_{\varepsilon}) := I(t).
	\end{align*}
	By \cref{estimates-U}\ref{point:estimates-U:1}, we see that, uniformly for $\varepsilon >0$ small,  $I(t) \to -\infty$ as $t \to + \infty$ and $I(t) \to F(u_c^+)$ as $t \to 0$ due to $\theta \to 1$. Hence, there exists $\varepsilon_0 >0$ and $ 0< t_0 < t_1 < \infty$ such that
	\begin{align*}
		\begin{split}
			F(\overline{w}_{\varepsilon, t})  < \gamma^+(c) + \dfrac{1}{6 \sqrt{a K_{GN}}} 
		\end{split}
	\end{align*}
	for $t \not\in [t_0, t_1]$ and $\varepsilon \in (0, \varepsilon_0]$. The lemma is proved.
\end{proof}

\begin{proof}[Proof of \cref{lemma:33}]	We assume throughout the proof that $ t \in [t_0, t_1]$.
	By using \eqref{LB}, we can write, 
	\begin{align}\label{L5}
		\begin{split}
			F(\overline{w}_{\varepsilon, t})  
			& =  \frac{1}{2}  A (w_{\varepsilon,t}) - \frac{\gamma}{4} \theta^{-3}B( w_{\varepsilon,t}) - \frac{a}{6} C (w_{\varepsilon,t})  \\
			& \leq \frac{1}{2}  A (w_{\varepsilon,t}) - \frac{\gamma}{4} \theta^{-3}\Big[B(u_c^+)  + B(t U_{\varepsilon}) + 4 \intrb\intrb \dfrac{\abs{u_c^+(x)}^2 u_c^+(y) (t U_{\varepsilon}(y))}{\abs{x-y}} dxdy \Big] - \frac{a}{6} C (w_{\varepsilon,t})  \\
			& = I_1 + I_2,
		\end{split}
	\end{align}
	where
	\begin{align*}
		I_1:= \frac{1}{2}  A (w_{\varepsilon,t})  - \frac{\gamma}{4} \Big[B(u_c^+)  + B(t U_{\varepsilon}) + 4 \intrb\intrb \dfrac{\abs{u_c^+(x)}^2 u_c^+(y) (t U_{\varepsilon}(y))}{\abs{x-y}} dxdy \Big] - \frac{a}{6} C (w_{\varepsilon,t}),
	\end{align*}
	and
	\begin{align}\label{Def-I2}
		I_2:= \dfrac{\gamma}{4}(1 - \theta^{-3}) \Big[B(u_c^+)  + B(t U_{\varepsilon}) + 4 \intrb\intrb \dfrac{\abs{u_c^+(x)}^2 u_c^+(y) (t U_{\varepsilon}(y))}{\abs{x-y}} dxdy \Big].
	\end{align}
	In view of \eqref{LA}, \eqref{LC} and using crucially \eqref{LE}, we have 
	\begin{align}\label{I_1}
		\begin{split}
			I_1  
			& \leq   \frac{1}{2} \Big[ A(u_c^+) + 2  \intrn \diff u_c^+(x) \cdot \diff (t U_{\varepsilon}(x)) \, dx  +    A(tU_{\varepsilon}) \Big] \\
			&  - \frac{\gamma}{4}  \Big[B(u_c^+)  + B(t U_{\varepsilon}) + 4 \intrb\intrb \dfrac{\abs{u_c^+(x)}^2 u_c^+(y) (t U_{\varepsilon}(y))}{\abs{x-y}} dxdy \Big]   \\
			& -   \frac{a}{6} \Big[  C(u_c^+) + C(t U_{\varepsilon}) + 6 \intrn  (u_c^+(x))^5  (tU_{\varepsilon}(x)) dx + 6 \intrn  u_c^+(x)  (tU_{\varepsilon}(x))^5 dx    \Big]  \\
			& = F(u_c^+) + F(t U_{\varepsilon}) -\lambda_c^+ \intrn  u_c^+(x) (tU_{\varepsilon}(x)) dx  - a \intrn  u_c^+(x)  (tU_{\varepsilon}(x))^5 dx.
		\end{split}
	\end{align}
	Now, we shall evaluate $I_2$. By \eqref{norm-L2} and \cref{estimates-U}\ref{point:estimates-U:3}, we get that
	\begin{align} \label{E-theta}
		\begin{split}
		\theta^2 &= \dfrac{\normLp{ w_{\varepsilon, t}}{2}^2}{c} = 1 + \frac{2}{c} \intrn  u_c^+(x)  (tU_{\varepsilon}(x)) dx  + \frac{t^2}{c} \normLp{ U_{\varepsilon}}{2}^2 \\
		&= 1 + \frac{2}{c} \intrn  u_c^+(x)  (tU_{\varepsilon}(x)) dx  + \frac{t^2}{c} \[\omega \Big(\int_0^2 \xi(r) dr \Big) \var + O(\var^2) \] 
		= 1 + \frac{2}{c} \intrn  u_c^+(x)  (tU_{\varepsilon}(x)) dx  + O(\varepsilon).
		\end{split}
	\end{align}
	Note that, by \cref{estimates-U}\ref{point:estimates-U:3} and \cref{lemma:21},
	\begin{equation}\label{size}
	\intrn  u_c^+(x)  (tU_{\varepsilon}(x)) dx  \sim \normLp{ \diff U_{\varepsilon}}{1} = O(\varepsilon^{\frac{1}{2}}).
	\end{equation}
	Observing that the Taylor expansion of $(1+x)^{- \frac{3}{2}}$ around $x=0$ is given by $(1+x)^{-\frac{3}{2}} = 1 - \frac{3}{2}x + O(x^2)$, we get, in view of \eqref{E-theta} and \eqref{size}, that
	\begin{align} \label{evaluate-I2-1}
		\begin{split}
			1 -\theta^{-3} &= 1- (\theta^2)^{-\frac{3}{2}} 
			 = 1 - \[1 + \frac{2}{c} \intrn  u_c^+(x)  (tU_{\varepsilon}(x))  + O(\varepsilon) \]^{- \frac{3}{2}} \\
			& = 1 - \[1 - \frac{3}{c} \intrn  u_c^+(x)  (tU_{\varepsilon}(x))  + O(\varepsilon) \] 
			 = \frac{3}{c} \intrn  u_c^+(x)  (tU_{\varepsilon}(x))  + O(\varepsilon).
		\end{split}
	\end{align}
	Concerning the term $B(tU_{\varepsilon})$, in view of \eqref {eqn:2.1} and \cref{estimates-U}\ref{point:estimates-U:3}, we have
	\begin{align} \label{evaluate-I2-2}
		\begin{split}
			B(tU_{\varepsilon}) = t^4 B(U_{\varepsilon}) \leq  t^4 K_1 \normLp{U_{\varepsilon}}{\frac{12}{5}}^4 = 
			t^4 K_1 \Big(\normLp{U_{\varepsilon}}{\frac{12}{5}}^\frac{12}{5}\Big)^{\frac{10}{6}} =  t^4 K_1 \Big(K_2 \varepsilon^{\frac{6}{5}} + o(\varepsilon^\frac{6}{5})\Big)^{\frac{10}{6}} = O(\varepsilon^2).
		\end{split}
	\end{align}
Also, using the Hardy-Littlewood-Sobolev inequality \eqref{ineq:HLS}, \cref{lemma:21} and \cref{estimates-U}\ref{point:estimates-U:3} we have 
	\begin{align}\label{needed}
		\intrb\intrb \dfrac{\abs{u_c^+(x)}^2 u_c^+(y)  (tU_{\varepsilon}(y))}{\abs{x-y}} dxdy &\leq K_2 \normLp{u_c^+}{\frac{12}{5}}^2 \normLp{u_c^+ U_{\varepsilon}}{\frac{6}{5}}
		\leq K_3  \normLp{U_{\varepsilon}}{\frac{6}{5}} = O(\varepsilon^{\frac{3}{5}}).
	\end{align}
From \eqref{evaluate-I2-2} and  \eqref{needed} we deduce that
	\begin{equation}\label{allneeded}
	 B(u_c^+)  + B(t U_{\varepsilon}) + 4 \intrb\intrb \dfrac{\abs{u_c^+(x)}^2 u_c^+(y) (t U_{\varepsilon}(y))}{\abs{x-y}} dxdy = B(u_c^+) + O (\varepsilon^{\frac{3}{5}}).
	\end{equation}
Taking into account, see \eqref{key-point}, that
	\begin{align*}
		c \lambda_c^+ = \frac{3\gamma}{4}B(u_c^+)
	\end{align*}
we obtain, combining \eqref{Def-I2}, \eqref{size}, \eqref{evaluate-I2-1} and \eqref{allneeded},  the following evaluation of $I_2$
	\begin{align} \label{I_2}
		\begin{split}
		I_2 \leq \dfrac{3 \gamma}{4c} B(u_c^+) \intrn  u_c^+(x)  (tU_{\varepsilon}(x)) dx + O(\varepsilon) 
		=   \lambda_c^+ \intrn  u_c^+(x)  (tU_{\varepsilon}(x)) dx   + O(\varepsilon).
		\end{split}
	\end{align}
	At this point, in view of \eqref{L5}, \eqref{I_1} and \eqref{I_2} we deduce that	
	\begin{align}\label{L8}
		\begin{split}
		F(\overline{w}_{\varepsilon, t})  
		&\leq   F(u_c^+) + F(t U_{\varepsilon}) - a \intrn  u_c^+(x)  (tU_{\varepsilon}(x))^5 dx  + O(\varepsilon) \\
		&\leq   F(u_c^+) + F(t U_{\varepsilon}) - a t_0^5 \intrn  u_c^+(x)  (U_{\varepsilon}(x))^5 dx  + O(\varepsilon).
	\end{split}
	\end{align}
	In view of \cref{estimates-U}\ref{point:estimates-U:1}, a direct calculation shows that 
	\begin{align}\label{againuseful}
		\begin{split}
			\max_{t \in [t_0, t_1]} F (t U_{\varepsilon}) &= \max_{t \in [t_0, t_1]} \[\dfrac{1}{2}A(t U_{\varepsilon}) - \dfrac{\gamma}{4} B(t U_{\varepsilon}) - \dfrac{a}{6} C(t U_{\varepsilon})\] \\
			&\leq \max_{t \in [t_0, t_1]} \[\dfrac{1}{2}A(t U_{\varepsilon}) - \dfrac{a}{6} C(t U_{\varepsilon})\] 
			\leq \max_{t >0} \[\dfrac{1}{2}A(t U_{\varepsilon}) - \dfrac{a}{6} C(t U_{\varepsilon})\] 
			= \dfrac{1}{3 \sqrt{a K_{GN}}} + O(\varepsilon). 
		\end{split}
	\end{align} 
	In view of \eqref{L8} and \eqref{againuseful}, by \cref{estimates-U}\ref{point:estimates-U:3} and \cref{lemma:21}, we conclude by observing that
	\begin{align*}
		- a t_0^5 \intrn  u_c^+(x)  (U_{\varepsilon}(x))^5 dx  \sim - \normLp{U_{\varepsilon}(x)}{5}^5 = - K \varepsilon^{\frac{1}{2}} + o(\varepsilon^{\frac{1}{2}}).
	\end{align*}
\end{proof}
\begin{proof}[Proof of \cref{theorem:1} in the critical case $p=6$ for $\gamma^-(c)$.]
We conclude that $\gamma^-(c)$ is reached by combining \cref{lemma:9}, \cref{lemma:19} and \cref{lemma:20}. The rest of the proof is identical to the one in the case $p \in (\frac{10}{3},6)$.
\end{proof}

\subsection{The compactness of any minimizing sequence associated to $\gamma^+(c)$ for $ p \in (\frac{10}{3}, 6]$}\label{Subsection3-4}

In this subsection we give the proof of \cref{th-minimization}. For short we introduce the following notations,
	\begin{align}\label{eq-notations}
	M:= \dfrac{p}{a \sigma (\sigma - 1) K_{GN}}, \quad N:= \dfrac{4(\sigma - 2)}{\gamma (\sigma - 1) K_{H}}, \quad k_0:= N^{-2}, \quad \mbox{and} \quad k_1 := k_0 c_1^3.
	\end{align}
	Note that
	\begin{align*}
	c_1 = N^{\frac{3p-10}{4(p-3)}} M^{\frac{1}{2(p-3)}}.
	\end{align*}
		
\begin{lemma} \label{lemma:27}
	Let  $p \in (\frac{10}{3},6]$ and $c \in (0, c_1)$. 
	\begin{enumerate}[label=(\roman*), ref=\roman*]
	\item \label{point:27i} If $u \in \slp$ then we have
	\begin{align}
		A(u) < k_0 c^3. \label{eqn:43}
	\end{align}
	\item \label{point:27ii}  $\slp \subset V(c)$ and
	\begin{align*}
		\gamma^+(c) = \inf_{u \in \slp}F(u) = \inf_{u \in \overline{V(c)}}F(u).
	\end{align*}
	\item \label{point:27iii} If $u_c$ is a minimizer for the minimization problem
	\begin{align*}
		\inf_{u \in \overline{V(c)}}F(u)
	\end{align*}
	then $u_c \in V(c)$ and $\gamma^+(c)$ is reached.
	\end{enumerate}
\end{lemma}
\begin{proof}
	\ref{point:27i}) Since $u \in \slp$, 
	\begin{align*}
		 A(u) = \dfrac{\gamma}{4} B(u) + \dfrac{a \sigma}{p} C(u) \quad \mbox{and} \quad A(u) > \dfrac{a \sigma (\sigma - 1)}{p} C(u).
	\end{align*}
	Using \cref{lemma:1}\eqref{point:1i}, we have
	\begin{align*}
	\dfrac{\sigma - 2}{\sigma - 1} A(u) < \dfrac{\gamma}{4} B(u) \leq \dfrac{\gamma}{4} K_{H}  \sqrt{A(u)} c^{\frac{3}{2}},
	\end{align*}
	which implies that
	\begin{align}
	A(u) < \[\dfrac{\gamma (\sigma - 1)  K_{H}}{4 (\sigma - 2)}\]^2  c^3 = N^{-2} c^3 = k_0 c^3 < k_0 c_1^3 = k_1. \label{eqn:40}
	\end{align}
	Hence, the point \eqref{point:27i} holds.
	
	\ref{point:27ii}) By \eqref{eqn:40}, we obtain that $\slp \subset V(c)$ and hence
	\begin{align*}
 		\inf_{u \in \slp}F(u) \geq \inf_{u \in \overline{V(c)}}F(u).
	\end{align*}
	To prove the point \eqref{point:27ii}, it is sufficient to show that
	\begin{align}
	\inf_{u \in \slp}F(u) \leq \inf_{u \in \overline{V(c)}}F(u). \label{eqn:42}
	\end{align}
	Firstly, we claim that $\slm \cap \overline{V(c)} = \emptyset$. Indeed, let $v \in \slm$. Taking into account that
	\begin{align*}
		A(v) < \dfrac{a \sigma (\sigma - 1)}{p} C(v),
	\end{align*}
	and using \cref{lemma:1}\eqref{point:1ii}, we obtain that
	\begin{align*}
		A(v) < \dfrac{a \sigma (\sigma - 1) }{p} K_{GN} c^{\frac{6-p}{4}} [A(v)]^{\frac{\sigma}{2}} = M^{-1} c^{\frac{6-p}{4}} [A(v)]^{\frac{\sigma}{2}}.
	\end{align*}
	This implies that
	\begin{align*}
		A(v) > M^{\frac{4}{3p-10}} c^{-\frac{6-p}{3p-10}}.
	\end{align*}
	By direct computations, we can check that
	\begin{align*}
		N^{-2} c_1^3 = M^{\frac{4}{3p-10}} c_1^{-\frac{6-p}{3p-10}},
	\end{align*}
	which implies that for all $0 < c < c_1$,
	\begin{align}
		k_1 = N^{-2} c_1^3 = M^{\frac{4}{3p-10}} c_1^{-\frac{6-p}{3p-10}} < M^{\frac{4}{3p-10}} c^{-\frac{6-p}{3p-10}} < A(v). \label{eqn:41}
	\end{align}
	Therefore, the claim holds. 
	Next, let $u \in S(c)$. Since the mapping $t \mapsto A(u^t)$ is continuous increasing, there exists a unique $t_u^1 > 0$ such that $A(u^{t_u^1}) = k_1$. By \cref{lemma:6} and \eqref{eqn:40}, \eqref{eqn:41}, we have
	\begin{align*}
		A(u^{s_u^+}) < A(u^{t_u^1}) < A(u^{s_u^-}),
	\end{align*}
	which implies that
	\begin{align*}
		s_u^+ < t_u^1 < s_u^-.
	\end{align*}
	Since $g_u'(t) > 0$ for all $t \in (s_u^+, s_u^-)$, we get that $g_u'(t) > 0$ for all $t \in (s_u^+, t_u^1]$ and hence
	\begin{align}
		F(u^{s_u^+}) = g_u(s_u^+) < g_u(t) = F(u^t) \qquad \forall t \in (s_u^+, t_u^1]. \label{eqn:52}
	\end{align}
	Since $\spu$ is the unique local minimum point for $g_u$ on $(0, s_u^-)$, we have that $F(u^{s_u^+}) \leq F(u^t)$ for all $t \in (0, t_u^1]$. Therefore, we obtain that
	\begin{align*}
		F(u^{s_u^+}) = \min\{F(u^t) | 0 < t \leq t_u^1 \} = \min\{F(u^t) | t\in \R, A(u^t) \leq k_1 \}.
	\end{align*}
	In particular, if $u \in \overline{V(c)}$ we have
	\begin{align*}
		F(u^{s_u^+}) = \min\{F(u^t) | t\in \R, A(u^t) \leq k_1 \} = \min\{F(u) | u \in \overline{V(c)} \} \leq F(u).
	\end{align*}
	This implies \eqref{eqn:42} and the point \eqref{point:27ii} is proved.
	
	\ref{point:27iii}) If we assume that $u_c \in \partial V(c)$, namely $A(u_c) = k_1$ and
	\begin{align*}
		F(u_c) = \min\{F(u) | u \in \overline{V(c)} \} = \min\{F(u^t) | t\in \R, A(u^t) \leq k_1 \},
	\end{align*}
	and we have a contradiction with \eqref{eqn:52}. Thus, we have $u_c \in V(c)$. Now, since the minimizer $u_c$ lies in the open (with respect to $S(c)$) set $V(c)$, we deduce from \cref{lemma:2.3} that $u_c \in \slc$. By $\slm \cap \overline{V(c)} = \emptyset$ and $\slz = \emptyset$, we conclude that $u_c \in \slp$ and thus $\gamma^+(c)$ is reached.
\end{proof}

\begin{lemma} \label{lemma:28}
	It holds that
	\begin{enumerate}[label=(\roman*), ref = \roman*]
		\item\label{point:28i} $\gamma^+(c) < 0$, $\forall c \in (0, c_1)$.
		\item\label{point:28ii} $c \in (0,c_1) \mapsto \gamma^+(c)$ is a continuous mapping.
		\item\label{point:28iii} Let $c \in (0,c_1)$, for all $\alpha \in (0, c)$, we have  $\gamma^+(c) \leq \gamma^+(\alpha) + \gamma^+(c - \alpha)$ and if $\gamma^+(\alpha)$ or $\gamma^+(c - \alpha)$ is reached then the inequality is strict.
	\end{enumerate}
\end{lemma}
\begin{proof}
	Point \eqref{point:28i} follows from \cref{lemma:16}. To prove \eqref{point:28ii}, let $c \in (0, c_1)$ be arbitrary and $(c_n) \subset (0, c_1)$ be such that $c_n \to c$. From the definition of $\gamma^+(c_n)$, for any $\varepsilon > 0$, there exists $u_n \in \Lambda^+(c_n)$ such that
	\begin{align}
		F(u_n) \leq \gamma^+(c_n) + \varepsilon. \label{eqn:44}
	\end{align}
	By \eqref{eqn:43}, we have $A(u_n) < k_0 c_n^3$. We set $y_n := \sqrt{\dfrac{c}{c_n}}\cdot u_n$. Hence, we have $y_n \in S(c)$ and
	\begin{align*}
		A(y_n) = \frac{c}{c_n} A(u_n) < \frac{c}{c_n} k_0 c_n^3 = k_0 c_n^2 c < k_0 c_1^3 = k_1.
	\end{align*}
	This implies that $y_n \in \overline{V(c)}$. Taking into account that $\dfrac{c}{c_n} \to 1$, we have
	\begin{align}
		\gamma^+(c) \leq F(y_n) = F(u_n) + o_n(1). \label{eqn:45}
	\end{align}
	Combining \eqref{eqn:44} and \eqref{eqn:45}, we get
	\begin{align*}
		\gamma^+(c) \leq \gamma^+(c_n) + \varepsilon + o_n(1).
	\end{align*}
	Reversing the argument we obtain similarly that
	\begin{align*}
		\gamma^+(c_n) \leq \gamma^+(c) + \varepsilon + o_n(1).
	\end{align*}
	Therefore, since $\varepsilon > 0$ is arbitrary, we deduce that $\gamma^+(c_n) \to \gamma^+(c)$. The point \eqref{point:28ii} follows.
	
	\ref{point:28iii}) Note that, fixed $\alpha \in (0,c)$, it is sufficient to prove that the following holds
	\begin{align}\label{L6-4}
		\forall \theta \in \(1, \frac{c}{\alpha}\] :  \gamma^+(\theta \alpha) \leq \theta \gamma^+(\alpha)
	\end{align}
	and that, if $\gamma^+(\alpha)$ is reached, the inequality is strict. Indeed, if \eqref{L6-4} holds then it follows directly that 
	\begin{align*}
	\gamma^+(c)=\frac{c-\alpha}{c} \gamma^+(c) + \frac{\alpha}{c} \gamma^+(c)
	=\frac{c-\alpha}{c} \gamma^+ \left( \frac{c}{c-\alpha}(c-\alpha) \right) + \frac{\alpha}{c} \gamma^+\left( \frac{c}{\alpha} \alpha \right)
	 \leq \gamma^+(c-\alpha) + \gamma^+(\alpha)
	\end{align*}
	with a strict inequality if $\gamma^+(\alpha)$ is reached. To prove that \eqref{L6-4} holds, note that for any $\varepsilon >0$ sufficiently small, there exist $u \in \Lambda^+(\alpha)$ such that
	\begin{align}
		F(u) \leq \gamma^+(\alpha) + \varepsilon. \label{eqn:46}
	\end{align} 
	By \eqref{eqn:43}, we have $A(u) < k_0 \alpha^3$. Consider now $v:= \sqrt{\theta} u$, we have
	\begin{align*}
		\normLp{v}{2}=\theta \normLp{u}{2}, \quad A(v)=\theta A(u), \quad B(v)=\theta^2 B(u), \quad C(v) = \theta^3 C(u).
	\end{align*}
	Therefore, we obtain that $v \in S(\theta\alpha)$ and 
	\begin{align*}
		A(v) = \theta A(u) < k_0 \theta \alpha^3 < k_0 (\theta \alpha)^3 \leq k_0 c^3 < k_1.
	\end{align*} 
	Hence, $v \in \overline{V(\theta\alpha)}$ and we can write 
	\begin{align*}
		\gamma^+(\theta\alpha) \leq F(v) &= \dfrac{1}{2} A(v) - \dfrac{\gamma}{4} B(v) - \dfrac{a}{p} C(v) 
		= \dfrac{1}{2} \theta A(u) - \dfrac{\gamma}{4} \theta^2 B(u) - \dfrac{a}{p} \theta^3 C(u) \\
		&< \dfrac{1}{2} \theta A(u) - \dfrac{\gamma}{4} \theta B(u) - \dfrac{a}{p} \theta C(u) 
		= \theta F(u) \leq \theta (\gamma^+(\alpha) + \varepsilon).
	\end{align*}
	Since $\varepsilon > 0$ is arbitrary, we have that  $\gamma^+(\theta \alpha) \leq \theta \gamma^+(\alpha)$. If $\gamma^+(\alpha)$ is reached then we can let $\varepsilon = 0$ in \eqref{eqn:46} and thus the strict inequality follows.
\end{proof}

\begin{lemma} \label{lemma:30}
	Let $(v_n) \subset \Hs{1}$ be such that $B(v_n) \to 0$ and $A(v_n) \leq k_1$. Then there exists a $b>0$ such that
	\begin{align}
		F(v_n) \geq b A(v_n) + o_n(1).
	\end{align}
\end{lemma}
\begin{proof}
	Indeed, using $B(v_n) \to 0$ and \cref{lemma:1}\eqref{point:1ii}, we have
	\begin{align*}
	F(v_n) &= \dfrac{1}{2} A(v_n) - \dfrac{a}{p} C(v_n) + o_n(1)
	\geq \dfrac{1}{2} A(v_n) - \dfrac{a}{p} K_{GN} c^{\frac{6-p}{4}} [A(v_n)]^{\frac{\sigma}{2}} + o_n(1)
	=b A(v_n) + o_n(1),
	\end{align*}
	where
	\begin{align*}
	b := \dfrac{1}{2} -  \limsup_{n \to \infty} \dfrac{a}{p} K_{GN} c^{\frac{6-p}{4}} [A(v_n)]^{\frac{\sigma}{2}-1} \geq \dfrac{1}{2} - \dfrac{a}{p} K_{GN} c_1^{\frac{6-p}{4}} k_1^{\frac{\sigma}{2}-1} = \dfrac{1}{2} - \dfrac{1}{\sigma(\sigma-1)}.
	\end{align*}
	Hence, $b > 0$ due to $\sigma > 2$. The lemma is proved.
\end{proof}

\begin{lemma} \label{lemma:29}
	For any $c \in (0, c_1)$, any minimizing sequence $(u_n)$ for $F$ on $\overline{V(c)}$ is, up to translation, strongly convergent in $\Hs{1}$. In addition all minimizers lies in $V(c)$. In particular $\gamma^+(c)$ is reached. 
\end{lemma}
\begin{proof}
	Since  $(u_n) \subset \overline{V(c)}$, it is bounded in $\Ho$. Also, from $\gamma^+(c) <0$ we deduce from  \cref{lemma:30} that there exists a $\beta_0 > 0$ and a sequence $(y_n) \subset \Rb$ such that
	\begin{align*}
		\int_{B(y_n, R)} \abs{u_n}^2 dx \geq \beta_0 > 0, \qquad \mbox{for some } R > 0.
	\end{align*} 
	This implies that 
	\begin{align*}
		u_n(x - y_n) \weakto u_c \neq 0 \quad \mbox{in } \Hs{1}, \quad \mbox{for some} \quad u_c \in \Hs{1}.
	\end{align*}
	Our aim is to prove that $w_n(x) := u_n(x - y_n) - u_c(x) \to 0$ in $\Hs{1}$. Clearly 
	\begin{align*}
	\normLp{u_n}{2}^2 &= \normLp{u_n(x-y_n)}{2}^2
	= \normLp{u_n(x-y_n) - u_c(x)}{2}^2 + \normLp{u_c}{2}^2 + o_n(1)
	=\normLp{w_n}{2}^2 + \normLp{u_c}{2}^2 + o_n(1).
	\end{align*}
	Thus, we have
	\begin{align}
		\normLp{w_n}{2}^2 = \normLp{u_n}{2}^2 -  \normLp{u_c}{2}^2 + o_n(1) = c - \normLp{u_c}{2}^2 + o_n(1). \label{eqn:47}
	\end{align}
	By the similar argument, 
	\begin{align}
		A(w_n) = A(u_n) -  A(u_c) + o_n(1). \label{eqn:48}
	\end{align}
	More generally, taking into account that any term in $F$ fulfills the splitting properties of Brezis-Lieb (see \cite{BrezisLieb1983} for terms $A$ and $C$; see \cite[Lemma 2.2]{ZHAO2008} or \cite[Proposition 3.1]{BellazziniSiciliano2011} for term $B$), we have
	\begin{align*}
		F(u_n - u_c) + F(u_c) =F(u_n) + o_n(1),
	\end{align*}
	and by the translational invariance, we obtain 
	\begin{align}
	\begin{split}
		F(u_n) =F(u_n(x-y_n)) =F(u_n(x-y_n) - u_c(x)) + F(u_c) + o_n(1)
			=F(w_n) + F(u_c) + o_n(1).
	\end{split} \label{eqn:49}
	\end{align}
	Now, we claim that
	\begin{align}
		\normLp{w_n}{2}^2 \to 0 \quad \text{as} \quad n \to \infty. \label{eqn:50}
	\end{align}
	In order to prove this, let us denote $\tilde{c}:= \normLp{u_c}{2}^2 > 0$. By \eqref{eqn:47}, if we show that $\tilde{c} = c$ then the claim follows. We assume by contradiction that $\tilde{c} < c$.
	In view of \eqref{eqn:47} and \eqref{eqn:48}, for $n$ large enough, we have $\normLp{w_n}{2}^2 \leq c$ and $A(w_n) \leq A(u_n) \leq k_1$. Hence, we obtain that $w_n \in \overline{V(\normLp{w_n}{2}^2)}$ and $F(w_n) \geq \gamma^+\(\normLp{w_n}{2}^2\)$.
	Recording that $F(u_n) \to \gamma^+(c)$, in view of \eqref{eqn:49}, we have
	\begin{align*}
		\gamma^+(c) = F(w_n) + F(u_c) \geq \gamma^+\(\normLp{w_n}{2}^2\) + F(u_c). 
	\end{align*}	
	Since the map $c \mapsto \gamma^+(c)$ is continuous (see \cref{lemma:28}\eqref{point:28ii}) and in view of \eqref{eqn:47}, we deduce that
	\begin{align}
		\gamma^+(c) \geq \gamma^+(c-\tilde{c}) + F(u_c). \label{eqn:51}
	\end{align}
	We also have that $u_c \in \overline{V(\tilde{c})}$ by the weak limit. This implies that $F(u_c) \geq \gamma^+(\tilde{c})$. If $F(u_c) > \gamma^+(\tilde{c})$, then it follows from \eqref{eqn:51} and \cref{lemma:28}\eqref{point:28iii} that
	\begin{align*}
		\gamma^+(c) > \gamma^+(c-\tilde{c}) + \gamma^+(\tilde{c}) \geq \gamma^+(c -\tilde{c} + \tilde{c}) = \gamma^+(c),
	\end{align*}
	which is impossible. Hence, we have $F(u_c) = \gamma^+(\tilde{c})$, namely $u_c$ is local minimizer on $\overline{V(\tilde{c})}$. So, we can using \cref{lemma:28}\eqref{point:28iii} with the strict inequality and we deduce from \eqref{eqn:51} that
	\begin{align*}
		\gamma^+(c) \geq \gamma^+(c-\tilde{c}) + F(u_c) = \gamma^+(c-\tilde{c}) + \gamma^+(\tilde{c}) > \gamma^+(c -\tilde{c} + \tilde{c}) = \gamma^+(c),
	\end{align*}
	which is impossible. Thus, the claim follows and $\normLp{u_c}{2}^2 = c$.\medskip
	
	Let us now show that $A(w_n) \to 0$. This will complete the proof of the lemma.
	In this aim first observe that since $(w_n)$ is a bounded sequence in $H^1(\R^N)$ we have, using \cref{lemma:1}\eqref{point:1i}, not only that $\normLp{w_n}{2}^2 \to 0$ but also that $B(w_n) \to 0$.  Now we remember that
	\begin{equation}\label{LLLL}
		F(u_n) = F(u_c) + F(w_n) + o_n(1) \to \gamma^+(c).
	\end{equation}
	Since $u_c \in \overline{V(c)}$ by weak convergence property, we have, by \cref{lemma:27}(ii), that  $F(u_c) \geq \gamma^+(c)$. Thus from \eqref{LLLL} we deduce, on one hand, that necessarily $F(w_n) \leq o(1)$. 
	On the other hand, since $A(w_n) \leq A(u_n) \leq k_1$,  \cref{lemma:30} implies that $F(w_n) \geq b A(w_n) + o_n(1)$ for some $b > 0$. 
	Hence, we conclude  $A(w_n) \to 0$ and thus that $u_n \to u_c \in \overline{V(c)}$ strongly in $H^1(\R^3)$. Finally, by \cref{lemma:27}\eqref{point:27iii}, we have $u_c \in V(c)$ and $\gamma^+(c)$ is reached. The lemma is proved.
\end{proof}

\subsection{Asymptotic behavior of the Lagrange multipliers}\label{Subsection3-5}

\begin{lemma}\label{lemma:22}
Let $p \in (\frac{10}{3}, 6]$.
	There exist two constants $K_1 > 0$ and $K_2 > 0$ such that for any $c \in (0, c_1)$, if $\lambda^+_c$ is the Lagrange parameter associated to a solution $u^+_c$ lying at the level $\gamma^+(c)$ then we have 
	\begin{align*}
	\abs{\gamma^+(c)} \leq K_1 c^3 \qquad \text{and} \qquad \lambda^+_c \leq K_2 c^2.
	\end{align*}
\end{lemma}
\begin{proof}
	By \cref{lemma:27}\eqref{point:27i}, we have 
	\begin{align*}
		A(u^+_c) < N^{-2} c^3 = \[\dfrac{\gamma (\sigma - 1)  K_{H}}{4 (\sigma - 2)} \]^2 c^3.
	\end{align*}
	Hence, we can deduce from \cref{lemma:1}\eqref{point:1i} that
	\begin{align*}
	B(u^+_c) \leq K_{H}  \sqrt{A(u^+_c)} c^{\frac{3}{2}} <  \dfrac{\gamma (\sigma - 1)  K_{H}^2}{4 (\sigma - 2)} c^3.
	\end{align*}
	Therefore, we have 
	\begin{align*}
	\abs{\gamma^+(c)} = \abs{F(u^+_c)}
	&= \abs*{\dfrac{\sigma-2}{2\sigma} A(u^+_c) - \dfrac{\gamma(\sigma-1)}{4\sigma} B(u^+_c)}
	\leq \dfrac{\sigma-2}{2\sigma} A(u^+_c) + \dfrac{\gamma(\sigma-1)}{4\sigma} B(u^+_c)\\
	&< \dfrac{\sigma-2}{2\sigma} \[\dfrac{\gamma (\sigma - 1)  K_{H}}{4 (\sigma - 2)}\]^2  c^3 + \dfrac{\gamma(\sigma-1)}{4\sigma}  \dfrac{\gamma (\sigma - 1)  K_{H}^2}{4 (\sigma - 2)} c^3 \\
	&= \dfrac{3 \gamma^2 (\sigma - 1)^2 K_{H}^2  }{32\sigma (\sigma - 2)} c^3 
	:= K_1 c^3.
	\end{align*}
	We deduce from \eqref{eqn:lambda} that
	\begin{align*}
	2(3p-6) c \lambda^+_c &= 2(6-p) A(u^+_c) + (5p -12)\gamma B(u^+_c) \\
	&< 2(6-p) \[\dfrac{\gamma (\sigma - 1)  K_{H}}{4 (\sigma - 2)}\]^2  c^3 + (5p -12) \gamma \dfrac{\gamma (\sigma - 1)  K_{H}^2}{4 (\sigma - 2)} c^3.
	\end{align*}
	This implies that there exists a constant $K_2 > 0$ such that $\lambda^+_c \leq K_2 c^2$.
	The lemma is proved.
\end{proof}


\begin{lemma}\label{lemma:23t}
	Let $p \in (\frac{10}{3}, 6)$. There exist two constants $K_1 > 0$ and $K_2 > 0$ such that is $\lambda^-_c$ denotes the Lagrange parameter associated to a solution $u^-_c$ lying at the level $\gamma^-(c)$,
	\begin{align*}
	\abs{\gamma^-(c)} > K_1 c^{- \frac{6-p}{3p -10}} \qquad\text{and}\qquad \lambda^-_c > K_2 c^{- \frac{2p-4}{3p -10}}.
	\end{align*}
\end{lemma}
\begin{proof}
	By $u^-_c \in \slm$, we have
	\begin{align*}
	A(u^-_c) < \dfrac{a \sigma (\sigma - 1)}{p} C(u^-_c).
	\end{align*}
	Using \cref{lemma:1}\eqref{point:1ii}, we obtain that
	\begin{align*}
	A(u^-_c) < \dfrac{a \sigma (\sigma - 1)}{p} K_{GN} c^{\frac{6-p}{4}} [A(u^-_c)]^{\frac{\sigma}{2}},
	\end{align*}
	which implies that
	\begin{align*}
	A(u^-_c) > \[\dfrac{p}{a \sigma (\sigma - 1) K_{GN}} \]^{\frac{2}{\sigma-2}} c^{- \frac{6-p}{3p -10}}.
	\end{align*}
	We have that 
	\begin{align*}
	\abs{\gamma^-(c)} = \abs{F(u^-_c)} = \abs*{-\dfrac{1}{2}A(u^-_c) + \dfrac{a(\sigma - 1)}{p}C(u^-_c)}
	> \dfrac{\sigma -2}{2\sigma} A(u^-_c) 
	> \dfrac{\sigma -2}{2\sigma} \[\dfrac{p}{a \sigma (\sigma - 1) K_{GN}} \]^{\frac{\sigma-2}{2}} c^{- \frac{6-p}{3p -10}}
	:= K_1 c^{- \frac{6-p}{3p -10}}.
	\end{align*} 	
	We deduce from \eqref{eqn:lambda} that
	\begin{align*}
	\lambda^-_c &= \dfrac{1}{c} \dfrac{1}{2(3p-6)} \[2(6-p) A(u^-_c) + (5p -12)\gamma B(u^-_c)\] \\
	&> \dfrac{1}{c} \dfrac{6-p}{3p-6} A(u^-_c) 
	> \dfrac{1}{c} \dfrac{6-p}{3p-6} \[\dfrac{p}{a \sigma (\sigma - 1) K_{GN}} \]^{\frac{\sigma-2}{2}} c^{- \frac{6-p}{3p -10}}
	:= K_2 c^{- \frac{2p-4}{3p -10}}.
	\end{align*}
	The lemma is proved.
\end{proof}

\begin{lemma}\label{lemma:25}
	Let $p=6$. There exists a constant $K_1 > 0$ such that if $\lambda^-_c$ denote the Lagrange parameter associated to a solution $u^-_c$ lying at the level $\gamma^-(c)$ then we have
	\begin{align*}
	\gamma^-(c) \to \dfrac{1}{3 \sqrt{a K_{GN}}} \quad \text{as} \quad c \to 0 \quad \mbox{and} \quad \lambda^-_c \leq K_1 c^{\frac{1}{2}}.
	\end{align*}
\end{lemma}
\begin{proof}
	Since $F(u)$ restricted to $\Lambda(c)$ is coercive on $\Ho$ (see \cref{lemma:5})  we have that $A(u^-_c)$ remain bounded.
	We deduce from \eqref{eqn:lambda} and \cref{lemma:1}\eqref{point:1i} that
	\begin{align*}
	\lambda^-_c = \dfrac{1}{c} \dfrac{3\gamma}{4} B(u^-_c) \leq \dfrac{1}{c} \dfrac{3\gamma}{4} K_{H}  \sqrt{A(u^-_c)} c^{\frac{3}{2}} := K_1 c^{\frac{1}{2}}.
	\end{align*}	
	We have that $B(u^-_c) \to 0$ as $c \to 0$ due to $B(u^-_c) \leq K_{H}  \sqrt{A(u^-_c)} c^{\frac{3}{2}}$. Since $Q(u^-_c) = 0$, we have
	\begin{align*}
	A(u^-_c) = a C(u^-_c) + o_c(1),
	\end{align*}
	where $o_c(1) \to 0$ as $c \to 0$. Passing to the limit as $c \to 0$, up to subsequence we infer that
	\begin{align*}
	\lim_{c \to 0} A(u^-_c) = \lim_{c \to 0} a C(u^-_c) :=\ell \geq 0.
	\end{align*}
	Using \cref{lemma:1}\eqref{point:1ii}, we have
	\begin{align*}
	\ell = \lim_{c \to 0} a C(u^-_c) \leq \lim_{c \to 0} a K_{GN} [A(u^-_c)]^3 = a K_{GN} \ell^3.
	\end{align*}
	Therefore, either $\ell = 0$ or $\ell \geq (a K_{GN})^{-\frac{1}{2}}$. Using \cref{lemma:16}\eqref{point:2}, we ensure that $\ell \geq (a K_{GN})^{-\frac{1}{2}}$. Hence, we have 
	\begin{align*}
	\gamma^-(c) + o_c(1) = F(u^-_c) &= \dfrac{\sigma - 2}{2\sigma} A(u^-_c) - \dfrac{\gamma(\sigma - 1)}{4\sigma} B(u^-_c)
	= \dfrac{1}{3} A(u^-_c) + o_c(1) = \dfrac{1}{3}\ell + o_c(1) \geq \dfrac{1}{3 \sqrt{a K_{GN}}} + o_c(1),
	\end{align*}
	which implies that
	\begin{align*}
	\gamma^-(c) \geq \dfrac{1}{3 \sqrt{a K_{GN}}} \quad \text{as} \quad c \to 0.
	\end{align*}
	Recording \cref{lemma:20}, the lemma is proved.
\end{proof}

\subsection{The monotonicity of the map $c \mapsto \gamma^{-}(c)$}\label{Subsection3-6}


\begin{lemma}\label{lemma:10}
	When $p \in (\frac{10}{3}, 6]$, the function $c \mapsto \gamma^{-}(c)$ is strictly decreasing on $(0, c_1)$.
\end{lemma}
\begin{proof}
	Let $0 < c_2 < c_3 < c_1$, Since $\gamma^{-}(c_2)$ is reached, there exists $u \in S(c_2)$ such that $F(u) = \gamma^{-}(c_2)$. We define $v \in S(c_3)$ by $v(x) = \sqrt{\theta} u(\theta x)$ where $\theta = \sqrt{\frac{c_2}{c_3}} < 1$. 	
	By direct calculations we have
	\begin{equation}
		A(v) = A(u), \quad B(v) = \theta^{-3}B(u) \quad \mbox{and } C(v) = \theta^{\frac{p}{2}-3}C(u). \label{eqn:37T2}
	\end{equation}
	Now observe that, since $\theta < 1$, for all $t>0$, 
	\begin{align}
		\begin{split}
			F(v^t) = \dfrac{1}{2}t^2 A(v) - \dfrac{\gamma}{4} t B(v) - \dfrac{a}{p} t^{\sigma} C(v)
			 = \dfrac{1}{2}t^2 A(u) - \dfrac{\gamma}{4} t \theta^{-3} B(u) - \dfrac{a}{p} t^{\sigma} \theta^{\frac{p}{2}-3} C(u)
			<  F(u^t).
		\end{split} \label{eqn:38T2}
	\end{align}
	By \eqref{eqn:43} and \eqref{eqn:41}, we have that
	\begin{align*}
		A(u^{s_u^+}) < k_1 < A(v^{s_v^-}),
	\end{align*}
	and thus $s_u^+ < s_v^-$ due to $A(v) = A(u)$. Hence, we can deduce from \eqref{eqn:38T2} that
	\begin{align*}
		F(v^{s_v^-}) < \max_{s_u^+ < t}F(u^t) = F(u) = \gamma^{-}(c_2).
	\end{align*}
	This implies that $\gamma^{-}(c_3) < \gamma^{-}(c_2)$ and hence, the lemma is proved.
\end{proof}

\section{The case $\gamma > 0,$ $ a < 0$ and $p \in (\frac{10}{3}, 6]$.} \label{Section4}
Throughout this section, we assume that $\gamma > 0,$ $ a < 0$ and $p \in (\frac{10}{3}, 6]$.
\begin{lemma} \label{lemma:5.1}
	$F$ restricted to $S(c)$ is coercive on $\Ho$ and bounded from below.
\end{lemma}
\begin{proof}
	Let $u \in S(c)$. Using \cref{lemma:1}\eqref{point:1i}, we obtain 
	\begin{align*}
	F(u) = \dfrac{1}{2} A(u) - \dfrac{\gamma}{4} B(u) - \dfrac{a}{p} C(u)
	\geq \dfrac{1}{2} A(u)  - \dfrac{\gamma}{4} K_{H}  \sqrt{A(u)} c^{\frac{3}{2}} - \dfrac{a}{p} C(u).
	\end{align*}
	Since $\gamma > 0,$ $ a < 0$, this concludes the proof. 
\end{proof}
In what follows, we collect some basic properties of $m(c)$ defined in \eqref{eqn:0.1}.
\begin{lemma} \label{lemma:5.5}
	It holds that
	\begin{enumerate}[label=(\roman*), ref = \roman*]
		\item\label{point:5.5i} $m(c) < 0$, $\forall c > 0$.
		\item\label{point:5.5ii} $c \mapsto m(c)$ is a continuous mapping.
		\item\label{point:5.5iii} For any $c_2 > c_1 > 0$, we have $c_1 m(c_2) \leq c_2 m(c_1)$. If $m(c_1)$ is reached then the inequality is strict.
		\item\label{point:5.5iv} For any $c_2, c_1 > 0$, we have $m(c_1+ c_2) \leq m(c_1) + m(c_2)$. If $m(c_1)$ or $m(c_2)$ is reached then the inequality is strict.
	\end{enumerate}
\end{lemma}
\begin{proof}
	\ref{point:5.5i}) For any $u \in S(c)$, we recall that $u^t \in S(c)$ and 
	$$g_u(t)= F(u^t) = \dfrac{1}{2}t^2 A(u) - \dfrac{\gamma}{4} t B(u) - \dfrac{a}{p} t^{\sigma} C(u) \quad \mbox{and also} \quad g_u'(t) = t A(u) - \dfrac{\gamma}{4} B(u) - \dfrac{a \sigma}{p} t^{\sigma - 1} C(u).$$
	We observe that $g_u(t) \to 0$ and $g_u'(t) \to - \frac{\gamma}{4} B(u) < 0$ as $t \to 0$. Therefore, there exists $t_0 > 0$ such that $F(u^{t_0}) = g_u(t_0) < 0$. Thus, we have $m(c) < 0$.
	

	\ref{point:5.5ii}) 
	We assume that $c_n \to c$. From the definition of $m(c_n)$, for any $\varepsilon > 0$, there exists $u_n \in S(c_n)$ such that
	\begin{align}
	F(u_n) \leq m(c_n) + \varepsilon. \label{eqn:5.9}
	\end{align}
	We set $y_n := \sqrt{\frac{c}{c_n}}\cdot u_n$. Taking into account that $y_n \in S(c)$ and $\frac{c}{c_n} \to 1$, we have
	\begin{align}
	m(c) \leq F(y_n) = F(u_n) + o_n(1).  \label{eqn:5.10}
	\end{align}
	Combining \eqref{eqn:5.9} and \eqref{eqn:5.10}, we get
	\begin{align*}
	m(c) \leq m(c_n) + \varepsilon + o_n(1).
	\end{align*}
	Reversing the argument we obtain similarly that
	\begin{align*}
	m(c_n) \leq m(c) + \varepsilon + o_n(1).
	\end{align*}
	Therefore, since $\varepsilon > 0$ is arbitrary, we deduce that $m(c_n) \to m(c)$. The point \eqref{point:5.5ii} follows.
	
	\ref{point:5.5iii}) Let $t := \frac{c_2}{c_1} > 1$. For any $\varepsilon > 0$, there exist $u \in S(c_1)$ such that
	\begin{align}
	F(u) \leq m(c_1) + \varepsilon. \label{eqn:5.11}
	\end{align}
	Let $v:= u(t^{-\frac{1}{3}}x)$. Then we have $\normLp{v}{2}^2 = t \normLp{u}{2}^2 = c_2$, hence $v \in S(c_2)$. Moreover, we have
	\begin{align*}
	A(v) = t^{\frac{1}{3}}A(u), \qquad B(v) = t^{\frac{5}{3}} B(u), \qquad C(v) = tC(u).
	\end{align*}
	Therefore, we have 
	\begin{align*}
	m(c_2) \leq F(v) &= \dfrac{1}{2} A(v) - \dfrac{\gamma}{4} B(v) - \dfrac{a}{p} C(v)
	= \dfrac{1}{2} t^{\frac{1}{3}}A(u) - \dfrac{\gamma}{4} t^{\frac{5}{3}} B(u) - \dfrac{a}{p} tC(u)\\
	&< \dfrac{1}{2} t A(u) - \dfrac{\gamma}{4} t B(u) - \dfrac{a}{p} t C(u) 
	= t \(\dfrac{1}{2} A(u) - \dfrac{\gamma}{4} B(u) - \dfrac{a}{p} C(u)\) \\
	&=t F(u) 
	\leq t(m(c_1) + \varepsilon) 
	= \dfrac{c_2}{c_1} m(c_1) + \frac{c_2 }{c_1}\varepsilon.
	\end{align*}
	Since $\varepsilon > 0$ is arbitrary, we have $c_1 m(c_2) \leq c_2 m(c_1)$. If $m(c_1)$ is reached then we can let $\varepsilon = 0$ in \eqref{eqn:5.11} and thus the strict inequality follows.
	
	\ref{point:5.5iv}) Assume first that $0 < c_1 \leq c_2$. Then, by \eqref{point:5.5iii}, we have that
	\begin{align*}
	m(c_1 + c_2) \leq \dfrac{c_1 + c_2}{c_2} m(c_2) = m(c_2) + \dfrac{c_1}{c_2} m(c_2) \leq m(c_2) + \dfrac{c_1}{c_2} \dfrac{c_2}{c_1} m(c_1) = m(c_1) + m(c_2).
	\end{align*}
	If $m(c_1)$ or $m(c_2)$ is reached, then we can use the strict inequality in \eqref{point:5.5iii} and thus the strict inequality follows.
	The case $0< c_2 < c_1$ can be treated reversing the role of $c_1$ and $c_2$. 
\end{proof}

\begin{lemma} \label{lemma:5.3}
	Let $(u_n) \subset S(c)$ be any minimizing sequence for $m(c)$. Then, there exist a $\beta_0 > 0$ and a sequence $(y_n) \in \Rb$ such that
	\begin{align}
	\int_{B(y_n, R)} \abs{u_n}^2 dx \geq \beta_0 > 0, \qquad \text{for some } R > 0. \label{eqn:5.1}
	\end{align}
\end{lemma}
\begin{proof}
	Since $F$ restricted to $S(c)$ is coercive on $\Ho$ (see \cref{lemma:5.1}), the sequence $(u_n)$ is bounded.
	Now, we assume that \eqref{eqn:5.1} does not hold. By \cite[Lemma I.1]{LIONS1984part2}, we have, for $q \in (2,6)$, $\normLp{u_n}{q} \to 0,$ as
	$n \to \infty.$
	This implies that
	\begin{align*}
		B(u_n) \leq K_1 \norm{u}_{\Lp{\frac{12}{5}}}^4 \to 0,
	\end{align*}
	due to \eqref{eqn:2.1}. Hence, we obtain	
	\begin{align*}
	F(u_n) = \dfrac{1}{2} A(u_n) - \dfrac{\gamma}{4} B(u_n) - \dfrac{a}{p} C(u_n) \to \dfrac{1}{2} A(u_n) - \dfrac{a}{p} C(u_n) \geq 0,
	\end{align*}
	due to $a < 0$. This contradicts $F(u_n) \to m(c) < 0$, see \cref{lemma:5.5}\eqref{point:5.5i}. 
\end{proof}

\begin{lemma} \label{lemma:5.4}
	Any minimizing sequence $(u_n) \subset S(c)$ for $m(c)$ is, up to translation, strongly convergent in $\Ho$.
\end{lemma}
\begin{proof} 
	Since $F$ restricted to $S(c)$ is coercive on $\Ho$ (see \cref{lemma:5.1}), the sequence $(u_n)$ is bounded in $\Ho$. We deduce from the weak convergence in $\Ho$, the local compactness in $\Lp{2}$ and \cref{lemma:5.3} that
	\begin{align*}
	u_n(x - y_n) \weakto u_c \neq 0 \quad\text{in } \Ho.
	\end{align*}
	Our aim is to prove that $w_n(x) := u_n(x - y_n) - u_c(x) \to 0$ in $\Ho$.
	Now, taking into account that $F$ fulfills the following splitting properties of Brezis-Lieb type (see \cite{BrezisLieb1983} for terms $A$ and $C$; see \cite[Lemma 2.2]{ZHAO2008} or \cite[Proposition 3.1]{BellazziniSiciliano2011} for term $B$),
	\begin{align*}
	F(u_n - u_c) + F(u_c) =F(u_n) + o_n(1),
	\end{align*}
	and by the translational invariance, we obtain 
	\begin{align}
	\begin{split}
	F(u_n) &=F(u_n(x-y_n))
	=F(u_n(x-y_n) - u_c(x)) + F(u_c) + o_n(1)
	=F(w_n) + F(u_c) + o_n(1),
	\end{split} \label{eqn:5.5}
	\end{align}
	and
	\begin{align*}
	\normLp{u_n}{2}^2 &= \normLp{u_n(x-y_n)}{2}^2
	= \normLp{u_n(x-y_n) - u_c(x)}{2}^2 + \normLp{u_c}{2}^2 + o_n(1)
	=\normLp{w_n}{2}^2 + \normLp{u_c}{2}^2 + o_n(1).
	\end{align*}
	Thus, we have
	\begin{align}
	\normLp{w_n}{2}^2 = \normLp{u_n}{2}^2 -  \normLp{u_c}{2}^2 + o_n(1) = c - \normLp{u_c}{2}^2 + o_n(1). \label{eqn:5.6}
	\end{align}
	We claim that
	\begin{align}
	\normLp{w_n}{2}^2 \to 0 \quad \text{as} \quad n \to \infty. \label{eqn:5.12}
	\end{align}
	In order to prove this, let us denote $c_1:= \normLp{u_c}{2}^2 > 0$. By \eqref{eqn:5.6}, if we show that $c_1 = c$ then the claim follows. We assume by contradiction that $c_1 < c$.
	Recording that $F(u_n) \to m(c)$, in view of \eqref{eqn:5.5}, we have
	\begin{align*}
	m(c) = F(w_n) + F(u_c) \geq m\(\normLp{w_n}{2}^2\) + F(u_c). 
	\end{align*}
	Since the map $c \mapsto m(c)$ is continuous (see \cref{lemma:5.5}\eqref{point:5.5ii}) and \eqref{eqn:5.6}, we deduce that
	\begin{align}
	m(c) \geq m(c-c_1) + F(u_c). \label{eqn:5.13}
	\end{align}
	If $F(u_c) > m(c_1)$, then it follows from \cref{lemma:5.5}\eqref{point:5.5iv} that
	\begin{align*}
	m(c) > m(c-c_1) + m(c_1) \geq m(c -c_1 + c_1) = m(c),
	\end{align*}
	which is impossible. Hence, we have $F(u_c) = m(c_1)$, namely $u_c$ is global minimizer with respect to $c_1$. So, we can using \cref{lemma:5.5}\eqref{point:5.5iv} with the strict inequality and we deduce from \eqref{eqn:5.13} that
	\begin{align*}
	m(c) \geq m(c-c_1) + F(u_c) = m(c-c_1) + m(c_1) > m(c -c_1 + c_1) = m(c),
	\end{align*}
	which is impossible. Thus, the claim follows and $\normLp{u_c}{2}^2 = c$.
	
	At this point, since $w_n$ is a bounded sequence in $\Ho$ and by \cref{lemma:1}\eqref{point:1i}, we have
	\begin{align*}
		B(w_n) \leq K_{H}  \sqrt{A(w_n)} \normLp{w_n}{2}^{3} \to 0.
	\end{align*}
	Thus, we obtain that
	\begin{align}
	\liminf_{n \to \infty} F(w_n) = \liminf_{n \to \infty} \[ \dfrac{1}{2} A(w_n) - \dfrac{a}{p}C(w_n)\] \geq 0. \label{eqn:5.14}
	\end{align}
	On the other hand, since $\normLp{u_c}{2}^2 = c$, we deduce from \eqref{eqn:5.5} that
	\begin{align*}
	F(u_n) = F(w_n) + F(u_c) +  o_n(1) \geq F(w_n) + m(c) + o_n(1),
	\end{align*}
	and by $F(u_n) \to m(c)$, we have that
	\begin{align}
	\limsup_{n \to \infty} F(w_n) \leq 0. \label{eqn:5.15}
	\end{align}
	Combining \eqref{eqn:5.14} and \eqref{eqn:5.15}, we obtain that $F(w_n) \to 0$. Hence, by \eqref{eqn:5.14} and by $a <0$, we have $A(w_n) \to 0$ and $C(w_n) \to 0$. Thus, we get $w_n \to 0$ in $\Ho$. The lemma is completed.
\end{proof}

\begin{proof}[Proof of \cref{theorem:2}]
	The proof follows directly from \cref{lemma:5.4} for the convergence of the minimizing sequence and from \cref{lemma:2.3} for the sign of the Lagrange parameter. 
\end{proof}

\begin{lemma}\label{lemma:5.6}
	There exist three constants $K_1, K_2, K_3 > 0$ such that if $\lambda_c$ denote the Lagrange parameter associated to a solution $u_c$ lying at the level $m(c)$ then we have
	\begin{align*}
		|m(c)| \leq K_1 c^3 + K_2 c^{2p-3}   \quad \mbox{and} \quad \lambda_c \leq K_3 c^{2}.
	\end{align*}
\end{lemma}
\begin{proof}
	By the fact that $m(c) < 0$ and by using \cref{lemma:1}\eqref{point:1i}, we get that
	\begin{align*}
		0 > m(c) = F(u_c) = \dfrac{1}{2} A(u_c) - \dfrac{\gamma}{4} B(u_c) - \dfrac{a}{p} C(u_c) \geq \dfrac{1}{2} A(u_c) - \dfrac{\gamma}{4} K_H \sqrt{A(u_c)} c^{\frac{3}{2}} - \dfrac{a}{p} C(u_c) \geq \dfrac{1}{2} A(u_c) - \dfrac{\gamma}{4} K_H \sqrt{A(u_c)} c^{\frac{3}{2}},				
	\end{align*}
 	due to our assumption $\gamma > 0$ and $a < 0$. This implies that
 	\begin{align*}
 		\sqrt{A(u_c)} < \dfrac{\gamma K_H}{2} c^{\frac{3}{2}}.
 	\end{align*}
 	Therefore, using again \cref{lemma:1}, we obtain that
 	\begin{align*}
 		|m(c)| = |F(u_c)| &= \Big|\dfrac{1}{2} A(u_c) - \dfrac{\gamma}{4} B(u_c) - \dfrac{a}{p} C(u_c) \Big| \leq \dfrac{1}{2} A(u_c) + \dfrac{\gamma}{4} B(u_c) - \dfrac{a}{p} C(u_c)\\
 		&\leq \dfrac{1}{2} A(u_c) + \dfrac{\gamma K_H}{4} \sqrt{A(u_c)} c^{\frac{3}{2}} - \dfrac{a K_{GN}}{p} [A(u_c)]^{\frac{\sigma}{2}} c^{\frac{6-p}{4}} \\
 		&\leq \dfrac{\gamma^2 K_H^2}{8} c^3 - \dfrac{a K_{GN}}{p} \[\dfrac{\gamma K_H}{2}\]^{\sigma} c^{\frac{3\sigma}{2}} c^{\frac{6-p}{4}} := K_1 c^3 + K_2 c^{2p-3}. 
 	\end{align*}
 	We deduce from \eqref{eqn:lambda} that
 	\begin{align*}
 		\lambda_c &= \dfrac{6-p}{3p-6} \dfrac{1}{c} A(u_c) + \dfrac{\gamma(5p-12)}{2(3p-6)} \dfrac{1}{c} B(u_c) 
 		\leq \dfrac{6-p}{3p-6} \dfrac{1}{c} A(u_c) + \dfrac{\gamma(5p-12)K_H}{2(3p-6)} \dfrac{1}{c}  \sqrt{A(u_c)} c^{\frac{3}{2}} \\
 		&\leq \dfrac{6-p}{3p-6}  \dfrac{1}{c} \dfrac{\gamma^2 K_H^2}{4}  c^{3} + \dfrac{\gamma(5p-12)K_H}{2(3p-6)} \dfrac{1}{c}  \dfrac{\gamma K_H}{2} c^{\frac{3}{2}} c^{\frac{3}{2}} := K_3 c^2.
 	\end{align*}
 	Thus, the lemma is proved.
\end{proof}

\section{The case $\gamma < 0,$ $ a > 0$ and $p=6$.} \label{Section:6}
Throughout this section, we assume that $\gamma < 0,$ $ a > 0$ and $p =6$. 
To prove the non-existence of the positive solution to \eqref{eqn:Laplace}, we first recall a Liouville-type result, see \cite[Theorem 2.1]{Armstrong-Sirakov-2011},
\begin{proposition}\label{theorem:Liouville-type}
	Assume that $N \geq 3$ and  the nonlinearity $f: (0, \infty) \mapsto (0, \infty)$ is
	continuous and satisfies 
	\begin{align*}
		\liminf_{s \to 0} s^{-\frac{N}{N-2}} f(s) > 0.
	\end{align*}
	Then the differential inequality  $- \laplace u \geq f(u)$ has no positive solution in any exterior domain of $\R^N$.
\end{proposition}

\begin{proof}[Proof of \cref{theorem:non-existence}]
	Let $u \in \Hs{1}$ be a non-trivial solution to \eqref{eqn:Laplace}. By \cref{lemma:2.3}, we have $\lambda < 0$ and $Q(u) = 0$. Hence, 
	\begin{align*}
		aC(u) = A(u) - \dfrac{\gamma}{4} B(u) > A(u)
	\end{align*}
	and using \cref{lemma:1}\eqref{point:1ii}, we obtain that
	\begin{align*}
		A(u) < aC(u) \leq a K_{GN} [A(u)]^{3}.
	\end{align*} 
	This implies that
	\begin{align*}
		A(u) > \sqrt{\dfrac{1}{a K_{GN}}}.
	\end{align*}
	Using again $Q(u) = 0$ we have that
	\begin{align*}
		F(u) = \dfrac{1}{2}A(u) - \dfrac{\gamma}{4} B(u) - \dfrac{a}{6} C(u) = \dfrac{5a}{6} C(u) - \dfrac{1}{2}A(u) > \dfrac{1}{3}A(u) > \dfrac{1}{3\sqrt{a K_{GN}}},
	\end{align*} 
	proving point \eqref{point:theorem:non:i}. To prove point \eqref{point:theorem:non:ii}, we assume by contradiction that there exists a positive solution $u \in \Hs{1}$ to \eqref{eqn:Laplace}. Then, by point \eqref{point:theorem:non:i}, the associated Lagrange multiplier  $\lambda $ is strictly negative. 
	In view of \eqref{eqn:tend-to-0}, there exists $R_0 > 0$ large enough such that 
	\begin{align*}
		(\abs{x}^{-1} * \abs{u}^2)(x) \leq - \dfrac{\lambda}{2 \gamma} \quad \mbox{for } |x| > R_0.
	\end{align*}
	Therefore, we get that
	\begin{align*}
		- \laplace u(x) = \( - \lambda + \gamma (\abs{x}^{-1} * \abs{u}^2)(x)  + a \abs{u(x)}^{4}\)u(x) \geq \( - \lambda + \gamma (\abs{x}^{-1} * \abs{u}^2)(x)\)u(x) \geq - \dfrac{\lambda}{2} u(x) \quad \mbox{for } |x| > R_0.
	\end{align*}
	By applying \cref{theorem:Liouville-type} with $f(s) = -\dfrac{\lambda}{2} s$, we obtain a contradiction, and thus point \eqref{point:theorem:non:ii} holds.		
\end{proof}

\begin{remark} \label{remrak:6.1}
In \cite[Theorem 1.2]{Soave2019Sobolevcriticalcase}, the author considers the equation
	\begin{align}
		-\laplace u - \lambda u - \mu \abs{u}^{q-2} u - \abs{u}^{2^*-2} u = 0 \quad \mbox{in } \R^N, \label{eqn:LaplaceS}
	\end{align}
	with $N \geq 3$, $2<q<2^*$ and $\mu < 0$. 
	If $u \in \mathit{H}^1(\R^N)$ is a non-trivial solution to \eqref{eqn:LaplaceS} then by \cite[Theorem 1.2]{Soave2019Sobolevcriticalcase}, the associated Lagrange multiplier $\lambda$ is positive and following the arguments in \cite[Proof of Theorem 1.2]{Soave2019Sobolevcriticalcase}, one obtains that
	\begin{align*}
		-\laplace u \geq \dfrac{\lambda}{2} u \quad \mbox{for } |x| > R_1,
	\end{align*}
	with $R_1 > 0$ large enough. Hence, by applying \cref{theorem:Liouville-type}, we see that \eqref{eqn:LaplaceS} has no positive solution $u \in {\mathit H}^{1}(\R^N)$ for all $N \geq 3$, improving slightly the conclusions of \cite[Theorem 1.2]{Soave2019Sobolevcriticalcase}.
	Actually, borrowing an observation from \cite{BerestyckiLions1983}, the
	non-existence results of \cite[Theorem 1.2]{Soave2019Sobolevcriticalcase} can be further extended by showing that \eqref{eqn:LaplaceS} has no non-trivial radial solutions in $\mathit{H}^1(\R^N)$ when $N \geq 3$ and $q > 2 + \frac{2}{N-1}$. Indeed, if $ u \in \mathit{H}^1(\R^N)$ is a radial function by \cite[Radial Lemma A.II]{BerestyckiLions1983}, there exist constants $C>0$ and $R_2 > 0$ such that
	\begin{align*}
		|u(x)| \leq C |x|^{-\frac{N-1}{2}}  \quad\mbox{for } |x| > R_2.
	\end{align*}
	Setting $V(x) = - \mu \abs{u(x)}^{q-2} - \abs{u(x)}^{2^*-2}$, we obtain that any radial solution  $u \in \mathit{H}^1(\R^N)$ satisfies
	\begin{align}
		-\laplace u(x) + V(x) u(x) = \lambda u(x), \label{eqn:6.3}
	\end{align}
	where, since $q > 2 + \frac{2}{N-1}$, 
	\begin{align*}
		\lim_{|x| \to \infty} |x| |V(x)| \leq \lim_{|x| \to \infty} \[-\mu C |x|^{-\frac{(N-1)(q-2)}{2} +1} + C |x|^{-\frac{(N-1)(2^*-2)}{2} +1} \] = 0.
	\end{align*}
	Then \eqref{eqn:6.3} has no solution in view of Kato's result \cite[page 404]{Kato-1959}, also see \cite{Simon-2019-Part2} which states that Schr\"{o}dinger operator 
	$ H = - \Delta + p(x)$ has no positive eigenvalue with an $L^2$-eigenfunction if $p(x) = o(|x|^{-1})$.
	\end{remark}
	
	\begin{remark}\label{open-problem}
	One may wonder if a non-existence result for radial solutions also holds for \eqref{eqn:Laplace} under the assumptions of \cref{theorem:non-existence}. The difficulty one faces is that, for any $u \in \mathit{H}^1(\R^N), $
	$(|x|^{-1}*|u|^2)(x) \geq C |x|^{-1}$ for $|x| > R$ for some $C, R > 0$ (see \cite{BellazziniJeanjeanLuo2013} or \cite[Appendix A.4]{Moroz-VanSchaftingen-2017}). Thus, the result of Kato used in \cref{remrak:6.1} cannot be directly applied and the non-existence of radial solutions to   \eqref{eqn:Laplace} when $\gamma <0$, $a>0$ and $p=6$ is an open problem.
\end{remark}

{\bf Acknowledgments:} The authors thank S. Cingolani for providing to us the statement and proof of \cref{lemma:3} and Colette De Coster for pointing to us \cite[Theorem 3.27]{Troianiello-1987}. 

\renewcommand{\bibname}{References}
\bibliographystyle{plain}
\bibliography{References}

\begin{thebibliography}{10}

\bibitem{AlvesJiMiyagaki2021}
Claudianor~O. Alves, Chao Ji, and Olimpio~H. Miyagaki.
\newblock Normalized solutions for a {S}chr\"odinger equation with critical
  growth in {$\Bbb{R}^N$}.
\newblock {\em arXiv.2102.03001}, 2021.

\bibitem{Armstrong-Sirakov-2011}
Scott~N. Armstrong and Boyan Sirakov.
\newblock Nonexistence of positive supersolutions of elliptic equations via the
  maximum principle.
\newblock {\em Comm. Partial Differential Equations}, 36(11):2011--2047, 2011.

\bibitem{BartschDevaleriola}
Thomas Bartsch and S\'{e}bastien de~Valeriola.
\newblock Normalized solutions of nonlinear {S}chr\"{o}dinger equations.
\newblock {\em Arch. Math. (Basel)}, 100(1):75--83, 2013.

\bibitem{BartschJeanjeanSoave16}
Thomas Bartsch, Louis Jeanjean, and Nicola Soave.
\newblock Normalized solutions for a system of coupled cubic {S}chr\"{o}dinger
  equations on {$\Bbb{R}^3$}.
\newblock {\em J. Math. Pures Appl. (9)}, 106(4):583--614, 2016.

\bibitem{Bartsch-Molle-Rizzi-Verzini-2021}
Thomas Bartsch, Riccardo Molle, Matteo Rizzi, and Gianmaria Verzini.
\newblock Normalized solutions of mass supercritical {S}chr\"odinger equations
  with potential.
\newblock {\em Communications in Partial Differential Equations, published
  Online}, 2021.

\bibitem{BartschSoave2019}
Thomas Bartsch and Nicola Soave.
\newblock Multiple normalized solutions for a competing system of
  {S}chr\"{o}dinger equations.
\newblock {\em Calc. Var. Partial Differential Equations}, 58(1):Paper No. 22,
  24, 2019.

\bibitem{Bartsch-Zhong-Zou-2021}
Thomas Bartsch, Xuexiu Zhong, and Wenming Zou.
\newblock Normalized solutions for a coupled {S}chr\"{o}dinger system.
\newblock {\em Math. Ann.}, 380(3-4):1713--1740, 2021.

\bibitem{BellazziniJeanjean2016}
Jacopo Bellazzini and Louis Jeanjean.
\newblock On dipolar quantum gases in the unstable regime.
\newblock {\em SIAM J. Math. Anal.}, 48(3):2028--2058, 2016.

\bibitem{BellazziniJeanjeanLuo2013}
Jacopo Bellazzini, Louis Jeanjean, and Tingjian Luo.
\newblock Existence and instability of standing waves with prescribed norm for
  a class of {S}chr\"{o}dinger-{P}oisson equations.
\newblock {\em Proc. Lond. Math. Soc. (3)}, 107(2):303--339, 2013.

\bibitem{BellazziniSiciliano2011-scaling}
Jacopo Bellazzini and Gaetano Siciliano.
\newblock Scaling properties of functionals and existence of constrained
  minimizers.
\newblock {\em J. Funct. Anal.}, 261(9):2486--2507, 2011.

\bibitem{BellazziniSiciliano2011}
Jacopo Bellazzini and Gaetano Siciliano.
\newblock Stable standing waves for a class of nonlinear
  {S}chr\"{o}dinger-{P}oisson equations.
\newblock {\em Z. Angew. Math. Phys.}, 62(2):267--280, 2011.

\bibitem{BerestyckiLions1983}
H.~Berestycki and P.-L. Lions.
\newblock Nonlinear scalar field equations. {I}. {E}xistence of a ground state.
\newblock {\em Arch. Rational Mech. Anal.}, 82(4):313--345, 1983.

\bibitem{BerestyckiLions1983_2}
H.~Berestycki and P.-L. Lions.
\newblock Nonlinear scalar field equations. {II}. {E}xistence of infinitely
  many solutions.
\newblock {\em Arch. Rational Mech. Anal.}, 82(4):347--375, 1983.

\bibitem{Bieganowski-Mederski2020}
Bartosz Bieganowski and Jarosław Mederski.
\newblock Normalized ground states of the nonlinear {S}chr\"odinger equation
  with at least mass critical growth.
\newblock {\em Journal of Functional Analysis, 280, (11), Article number
  108989}, 2021.

\bibitem{Brezis-2011}
Haim Brezis.
\newblock {\em Functional analysis, {S}obolev spaces and partial differential
  equations}.
\newblock Universitext. Springer, New York, 2011.

\bibitem{BrezisLieb1983}
Ha\"{\i}m Br\'{e}zis and Elliott Lieb.
\newblock A relation between pointwise convergence of functions and convergence
  of functionals.
\newblock {\em Proc. Amer. Math. Soc.}, 88(3):486--490, 1983.

\bibitem{BrezisNirenberg1983}
Ha\"{\i}m Br\'{e}zis and Louis Nirenberg.
\newblock Positive solutions of nonlinear elliptic equations involving critical
  {S}obolev exponents.
\newblock {\em Comm. Pure Appl. Math.}, 36(4):437--477, 1983.

\bibitem{Catto-Dolbeault-Sanchez-Soler-2013}
I.~Catto, J.~Dolbeault, O.~S\'{a}nchez, and J.~Soler.
\newblock Existence of steady states for the
  {M}axwell-{S}chr\"{o}dinger-{P}oisson system: exploring the applicability of
  the concentration-compactness principle.
\newblock {\em Math. Models Methods Appl. Sci.}, 23(10):1915--1938, 2013.

\bibitem{Cazenave2003semilinear}
Thierry Cazenave.
\newblock {\em Semilinear {S}chr\"{o}dinger equations}, volume~10 of {\em
  Courant Lecture Notes in Mathematics}.
\newblock New York University, Courant Institute of Mathematical Sciences, New
  York; American Mathematical Society, Providence, RI, 2003.

\bibitem{CazenaveLions1982}
Thierry Cazenave and Pierre-Louis Lions.
\newblock Orbital stability of standing waves for some nonlinear
  {S}chr\"{o}dinger equations.
\newblock {\em Comm. Math. Phys.}, 85(4):549--561, 1982.

\bibitem{CingolaniJeanjean2019}
Silvia Cingolani and Louis Jeanjean.
\newblock Stationary waves with prescribed {$L^2$}-norm for the planar
  {S}chr\"{o}dinger-{P}oisson system.
\newblock {\em SIAM J. Math. Anal.}, 51(4):3533--3568, 2019.

\bibitem{AprileMugnai2004_Non}
Teresa D'Aprile and Dimitri Mugnai.
\newblock Non-existence results for the coupled {K}lein-{G}ordon-{M}axwell
  equations.
\newblock {\em Adv. Nonlinear Stud.}, 4(3):307--322, 2004.

\bibitem{ghoussoub1993duality}
Nassif Ghoussoub.
\newblock {\em Duality and perturbation methods in critical point theory},
  volume 107 of {\em Cambridge Tracts in Mathematics}.
\newblock Cambridge University Press, Cambridge, 1993.
\newblock With appendices by David Robinson.

\bibitem{GouJeanjean2016}
Tianxiang Gou and Louis Jeanjean.
\newblock Existence and orbital stability of standing waves for nonlinear
  {S}chr\"{o}dinger systems.
\newblock {\em Nonlinear Anal.}, 144:10--22, 2016.

\bibitem{Gou-Zhang-2021}
Tianxiang Gou and Zhitao Zhang.
\newblock Normalized solutions to the {C}hern-{S}imons-{S}chr\"{o}dinger
  system.
\newblock {\em J. Funct. Anal.}, 280(5):108894, 65, 2021.

\bibitem{Ikoma2014}
Norihisa Ikoma.
\newblock Compactness of minimizing sequences in nonlinear {S}chr\"{o}dinger
  systems under multiconstraint conditions.
\newblock {\em Adv. Nonlinear Stud.}, 14(1):115--136, 2014.

\bibitem{JEANJEAN1997}
Louis Jeanjean.
\newblock Existence of solutions with prescribed norm for semilinear elliptic
  equations.
\newblock {\em Nonlinear Anal.}, 28(10):1633--1659, 1997.

\bibitem{JeanjeanJendrejLeVisciglia2020}
Louis Jeanjean, Jacek Jendrej, Thanh~Trung Le, and Nicola Visciglia.
\newblock Orbital stability of ground states for a {S}obolev critical
  {S}chr\"odinger equation.
\newblock {\em arXiv.2008.12084}, 2020.

\bibitem{JeanjeanLe2020}
Louis Jeanjean and Thanh~Trung Le.
\newblock Multiple normalized solutions for a sobolev critical {S}chr\"odinger
  equation.
\newblock {\em Math. Ann. (2021). https://doi.org/10.1007/s00208-021-02228-0}.

\bibitem{Lu3}
Louis Jeanjean and Sheng-Sen Lu.
\newblock A mass supercritical problem revisited.
\newblock {\em Calc. Var. Partial Differential Equations}, 59(5):Paper No. 174,
  43, 2020.

\bibitem{ZAMP2013}
Louis Jeanjean and Tingjian Luo.
\newblock Sharp nonexistence results of prescribed {$L^2$}-norm solutions for
  some class of {S}chr\"{o}dinger-{P}oisson and quasi-linear equations.
\newblock {\em Z. Angew. Math. Phys.}, 64(4):937--954, 2013.

\bibitem{Kato-1959}
Tosio Kato.
\newblock Growth properties of solutions of the reduced wave equation with a
  variable coefficient.
\newblock {\em Comm. Pure Appl. Math.}, 12:403--425, 1959.

\bibitem{Kikuchi-2007}
Hiroaki Kikuchi.
\newblock Existence and stability of standing waves for
  {S}chr\"{o}dinger-{P}oisson-{S}later equation.
\newblock {\em Adv. Nonlinear Stud.}, 7(3):403--437, 2007.

\bibitem{Li-Ma-2020}
Xinfu Li and Shiwang Ma.
\newblock Choquard equations with critical nonlinearities.
\newblock {\em Commun. Contemp. Math.}, 22(4):1950023, 28, 2020.

\bibitem{LiebLoss2001analysis}
Elliott~H. Lieb and Michael Loss.
\newblock {\em Analysis}, volume~14 of {\em Graduate Studies in Mathematics}.
\newblock American Mathematical Society, Providence, RI, second edition, 2001.

\bibitem{LIONS1984-1}
Pierre-Louis Lions.
\newblock The concentration-compactness principle in the calculus of
  variations. {T}he locally compact case. {I}.
\newblock {\em Ann. Inst. H. Poincar\'{e} Anal. Non Lin\'{e}aire},
  1(2):109--145, 1984.

\bibitem{LIONS1984part2}
Pierre-Louis Lions.
\newblock The concentration-compactness principle in the calculus of
  variations. {T}he locally compact case. {II}.
\newblock {\em Ann. Inst. H. Poincar\'{e} Anal. Non Lin\'{e}aire},
  1(4):223--283, 1984.

\bibitem{Luo-Yang-Yang-2021}
Xiao Luo, Tao Yang, and Xiaolong Yang.
\newblock Multiplicity and asymptotics of standing waves for the energy
  critical half-wave.
\newblock {\em arXiv.2102.09702}, 2021.

\bibitem{Moroz-VanSchaftingen-2017}
Vitaly Moroz and Jean Van~Schaftingen.
\newblock A guide to the {C}hoquard equation.
\newblock {\em J. Fixed Point Theory Appl.}, 19(1):773--813, 2017.

\bibitem{Nirenberg1985}
Louis Nirenberg.
\newblock On elliptic partial differential equations.
\newblock {\em Ann. Scuola Norm. Sup. Pisa Cl. Sci. (3)}, 13:115--162, 1959.

\bibitem{NorisTavaresVerzini2019}
Benedetta Noris, Hugo Tavares, and Gianmaria Verzini.
\newblock Normalized solutions for nonlinear {S}chr\"{o}dinger systems on
  bounded domains.
\newblock {\em Nonlinearity}, 32(3):1044--1072, 2019.

\bibitem{Pellaci-Pistoia-Vaira-Verzini-2021}
Benedetta Pellacci, Angela Pistoia, Giusi Vaira, and Gianmaria Verzini.
\newblock Normalized concentrating solutions to nonlinear elliptic problems.
\newblock {\em J. Differential Equations}, 275:882--919, 2021.

\bibitem{Pierotti-Verzini-2017}
Dario Pierotti and Gianmaria Verzini.
\newblock Normalized bound states for the nonlinear {S}chr\"{o}dinger equation
  in bounded domains.
\newblock {\em Calc. Var. Partial Differential Equations}, 56(5):Paper No. 133,
  27, 2017.

\bibitem{DavidRuiz2006}
David Ruiz.
\newblock The {S}chr\"{o}dinger-{P}oisson equation under the effect of a
  nonlinear local term.
\newblock {\em J. Funct. Anal.}, 237(2):655--674, 2006.

\bibitem{Sanchez-Soler-2004}
\'{O}scar S\'{a}nchez and Juan Soler.
\newblock Long-time dynamics of the {S}chr\"{o}dinger-{P}oisson-{S}later
  system.
\newblock {\em J. Statist. Phys.}, 114(1-2):179--204, 2004.

\bibitem{Simon-2019-Part2}
Barry Simon.
\newblock Tosio {K}ato's work on non-relativistic quantum mechanics, {P}art 2.
\newblock {\em Bull. Math. Sci.}, 9(1):1950005, 105, 2019.

\bibitem{Soave2019}
Nicola Soave.
\newblock Normalized ground states for the {NLS} equation with combined
  nonlinearities.
\newblock {\em J. Differential Equations}, 269(9):6941--6987, 2020.

\bibitem{Soave2019Sobolevcriticalcase}
Nicola Soave.
\newblock Normalized ground states for the {NLS} equation with combined
  nonlinearities: the {S}obolev critical case.
\newblock {\em J. Funct. Anal.}, 279(6):108610, 43, 2020.

\bibitem{Strauss1977}
Walter~A. Strauss.
\newblock Existence of solitary waves in higher dimensions.
\newblock {\em Comm. Math. Phys.}, 55(2):149--162, 1977.

\bibitem{Tarantello92}
G.~Tarantello.
\newblock On nonhomogeneous elliptic equations involving critical {S}obolev
  exponent.
\newblock {\em Ann. Inst. H. Poincar\'{e} Anal. Non Lin\'{e}aire},
  9(3):281--304, 1992.

\bibitem{Troianiello-1987}
Giovanni~Maria Troianiello.
\newblock {\em Elliptic differential equations and obstacle problems}.
\newblock The University Series in Mathematics. Plenum Press, New York, 1987.

\bibitem{Wei-Wu2021}
Juncheng Wei and Yuanze Wu.
\newblock Normalized solutions for {S}chr\"odinger equations with critical
  sobolev exponent and mixed nonlinearities.
\newblock {\em arXiv.2102.04030}, 2021.

\bibitem{Yao-Sun-Wu-2021}
Shuai Yao, Juntao Sun, and Tsung-fang Wu.
\newblock Normalized solutions for the {S}chr\"odinger equation with combined
  hartree type and power nonlinearities.
\newblock {\em arXiv:2102.10268}, 2021.

\bibitem{ZHAO2008}
Leiga Zhao and Fukun Zhao.
\newblock On the existence of solutions for the {S}chr\"{o}dinger-{P}oisson
  equations.
\newblock {\em J. Math. Anal. Appl.}, 346(1):155--169, 2008.

\end{thebibliography}
\vspace{0.25cm}
\end{document}